\documentclass[11pt]{article}

\usepackage{amsmath,amsthm,amssymb,xypic}

\theoremstyle{plain}
    \newtheorem{theorem}{Theorem}[section]
    \newtheorem{lemma}[theorem]{Lemma}
    \newtheorem{corollary}[theorem]{Corollary}
    \newtheorem{proposition}[theorem]{Proposition}

 \theoremstyle{definition}
    \newtheorem{definition}[theorem]{Definition}

    \newtheorem{remark}[theorem]{Remark}

\theoremstyle{remark}

\numberwithin{equation}{section}

 \DeclareMathOperator{\Tr}{Tr}
 \DeclareMathOperator{\tr}{tr}
  \DeclareMathOperator{\Str}{Str}

  \DeclareMathOperator{\Stab}{Stab}
\DeclareMathOperator{\Ad}{Ad}

\DeclareMathOperator{\ind}{index}

\DeclareMathOperator{\End}{End}
\DeclareMathOperator{\Hom}{Hom}

\DeclareMathOperator{\ch}{ch}
\DeclareMathOperator{\Todd}{Todd}

\DeclareMathOperator{\reg}{reg}

\DeclareMathOperator{\rank}{rank}
\DeclareMathOperator{\sign}{sign}
\DeclareMathOperator{\diam}{diam}

\DeclareMathOperator{\Spin}{Spin}

\DeclareMathOperator{\SO}{SO}
\DeclareMathOperator{\SL}{SL}

         \DeclareMathOperator{\supp}{supp}
\DeclareMathOperator{\pt}{pt}

  \DeclareMathOperator{\DInd}{D-Ind}
\DeclareMathOperator{\KInd}{K-Ind}
\DeclareMathOperator{\vol}{vol}
\begin{document}


\newcommand{\myemph}{\emph}

\newcommand{\Spinc}{\Spin^c}

    \newcommand{\R}{\mathbb{R}}
    \newcommand{\C}{\mathbb{C}} 
    \newcommand{\N}{\mathbb{N}}
    \newcommand{\Z}{\mathbb{Z}} 
    \newcommand{\Q}{\mathbb{Q}}
    \newcommand{\bT}{\mathbb{T}}
    \newcommand{\bP}{\mathbb{P}}

\newcommand{\g}{\mathfrak{g}}
\newcommand{\h}{\mathfrak{h}}
\newcommand{\p}{\mathfrak{p}}
\newcommand{\kg}{\mathfrak{g}} 
\newcommand{\kt}{\mathfrak{t}}
\newcommand{\ka}{\mathfrak{a}}
\newcommand{\XX}{\mathfrak{X}}
\newcommand{\kh}{\mathfrak{h}} 
\newcommand{\kp}{\mathfrak{p}}
\newcommand{\kk}{\mathfrak{k}}

\newcommand{\cA}{\mathcal{A}}
\newcommand{\cE}{\mathcal{E}}
\newcommand{\calL}{\mathcal{L}}
\newcommand{\calH}{\mathcal{H}}
\newcommand{\cO}{\mathcal{O}}
\newcommand{\cB}{\mathcal{B}}
\newcommand{\cK}{\mathcal{K}}
\newcommand{\cP}{\mathcal{P}}
\newcommand{\cN}{\mathcal{N}}
\newcommand{\calD}{\mathcal{D}}
\newcommand{\cC}{\mathcal{C}}
\newcommand{\calS}{\mathcal{S}}

\newcommand{\cCM}{\cC}
\newcommand{\PM}{P}
\newcommand{\DM}{D}
\newcommand{\LM}{L}
\newcommand{\vM}{v}

\newcommand{\sigDg}{\sigma^D_g}

\newcommand{\Bigwedge}{\textstyle{\bigwedge}}

\newcommand{\ii}{\sqrt{-1}}

\newcommand{\Ubar}{\overline{U}}

\newcommand{\beq}[1]{\begin{equation} \label{#1}}
\newcommand{\eeq}{\end{equation}}

\newcommand{\Todo}{\textbf{To do}}

\newcommand{\mattwo}[4]{
\left( \begin{array}{cc}
#1 & #2 \\ #3 & #4
\end{array}
\right)
}

\newenvironment{proofof}[1]
{\noindent \emph{Proof of #1.}}{\hfill $\square$}

\title{A fixed point formula and Harish-Chandra's character formula\footnote{\emph{2010 MSC:} 58J20 (primary), 19K56, 46L80, 22E46 (secondary); 
\emph{key words:} equivariant index, fixed point formula, $K$-theory, $C^*$-algebra, semisimple Lie group, character formula}}

\author{Peter Hochs\footnote{University of Adelaide, \texttt{peter.hochs@adelaide.edu.au}}, Hang Wang\footnote{University of Adelaide, \texttt{hang.wang01@adelaide.edu.au}}}

\date{\today}

\maketitle

\begin{abstract}
The main result in this paper is a fixed point formula for equivariant indices of elliptic differential operators, for proper actions by connected semisimple Lie groups on possibly noncompact manifolds, with compact quotients. For compact groups and manifolds, this reduces to 
 the Atiyah--Segal--Singer fixed point formula. Other special cases include an index theorem by Connes and Moscovici for homogeneous spaces, and an earlier index theorem by the second author, both in cases where the group acting is connected and semisimple. As an application of this fixed point formula, we give a new proof of Harish-Chandra's character formula for discrete series representations.
\end{abstract}

\tableofcontents

\section{Introduction}

Equivariant index theory has always had powerful applications to representation theory. In this paper, we prove a fixed point formula for indices of elliptic operators that are equivariant with respect to an action by a connected semisimple Lie group. We use this formula to give a new proof of Harish-Chandra's character formula for the discrete series. The arguments use $K$-theory as an essential ingredient.

Consider a Lie group $G$, acting properly on a manifold $M$. Let $D$ be an odd, self-adjoint, $G$-equivariant, elliptic differential operator on a $\Z_2$-graded, $G$-equivariant vector bundle over $M$. We denote the restrictions of $D$ to sections of the even and odd parts of this vector bundle by $D^+$ and $D^-$, respectively. Suppose for now that  $M$ and $G$ are compact. Then the kernels of $D^+$ and $D^-$ are finite-dimensional representations of $G$. The formal difference of their equivalence classes is the equivariant index of $D$:
\beq{eq index intro}
\ind_G(D) = [\ker(D^+)]-[\ker(D^-)].
\eeq
For an element $g \in G$, we can evaluate the characters of the representations on the right hand side at $g$, to obtain
\beq{eq index g intro}
\ind_G(D)(g) = \tr(\text{$g$ on $\ker(D^+)$})- \tr(\text{$g$ on $\ker(D^-)$}) \quad \in \C.
\eeq
The Atiyah--Segal--Singer fixed point formula (Theorem 2.12 in \cite{Atiyah68} and (3.1) in \cite{ASIII}) is an expression for this number in terms of geometric data near the fixed point set
\[
M^g = \{m \in M; gm = m\}.
\]

Atiyah and Bott showed in \cite{ABII} that this fixed point formula implies Weyl's character formula for irreducible representations of compact connected Lie groups. In fact, they used a slightly different fixed point formula, Theorem A in \cite{ABI}, which is equivalent to the Atiyah--Segal--Singer fixed point formula in the case they considered to prove Weyl's character formula. In this case, one takes $M = G/T$, for a maximal torus $T<G$. The Borel--Weil theorem makes it possible to realise an irreducible representation of $G$ as the equivariant index of a twisted Dolbeault--Dirac operator on $G/T$. For this operator, the fixed point formula precisely becomes Weyl's character formula.

If $M$ and $G$ are noncompact, but $M/G$ is compact, then one can still define an equivariant index of $D$, using the analytic assembly map from the Baum--Connes conjecture \cite{Connes94}. This takes values in the $K$-theory group $K_*(C^*_rG)$ of the reduced group $C^*$-algebra $C^*_rG$ of $G$, so one obtains
\beq{eq ass map intro}
\ind_G(D) \in K_*(C^*_rG).
\eeq
If $M$ and $G$ are compact, then this reduces to \eqref{eq index intro}.

We suppose that $G$ is connected and semisimple, and use Harish-Chandra's Schwartz algebra $\cC(G)$. This has the same $K$-theory as $C^*_rG$. We use traces
\[
\tau_g\colon \cC(G)\to \C,
\]
for elements $g \in G$, defined as orbital integrals
\[
\tau_g(f) = \int_{G/Z_G(g)} f(xgx^{-1})\, dx,
\]
for $f \in \cC(G)$. Applying such a trace to the index \eqref{eq ass map intro} yields a number
\beq{eq tau g index}
\tau_g(\ind_G(D)) \quad \in \C,
\eeq
generalising \eqref{eq index g intro}. If $g=e$, then this is the $L^2$-index as defined by Atiyah \cite{AtiyahL2}.

Using heat kernel methods, we generalise the Atiyah--Segal--Singer fixed point formula to derive a cohomological formula for \eqref{eq tau g index} in terms of fixed point data. This is the main result of this paper, Theorem \ref{thm fixed pt}. It generalises the main result in \cite{Wang14}, Proposition 6.11 in that paper,  from $g=e$ to more general elements $g \in G$, for actions by connected semisimple groups. This generalisation is essential for our application to representation theory. Proposition 6.11 in \cite{Wang14} in turn generalises Connes and Moscovici's index formula on homogeneous spaces, Theorem 5.1 in \cite{Connes82}. So Theorem \ref{thm fixed pt} also generalises that result, again for connected semisimple groups.

Theorem 6.1 in \cite{Wangwang} is a related fixed point formula for actions by discrete groups on orbifolds. The techniques used there are very different from our arguments for connected groups. Another fixed point formula on noncompact manifolds is Theorem 2.9 in \cite{HW}. The index used in \cite{HW} was defined differently, in terms of $KK$-theory, but it turns out to equal \eqref{eq tau g index} if $M^g$ is compact, because it satisfies  the same fixed point formula. An advantage of \eqref{eq tau g index} over the index used in \cite{HW} is that we can use the former index to directly realise character values of discrete series representations. That allows us to deduce Harish-Chandra's character formula from the fixed point formula. 

If $G$ has a compact Cartan subgroup $T$, then it has discrete series representations. If we take $M = G/T$, and let $D$ be the twisted Dolbeault--Dirac operator used by Schmid \cite{Schmid76} to realise the discrete series, then \eqref{eq tau g index} is the value of the character of a discrete series representation at a regular element $g \in T$. The fixed point formula now reduces to Harish-Chandra's character formula for the discrete series. Because this character formula plays a central role in Harish-Chandra's discussion of the discrete series, we need to be careful about the statement and proof of this result; where necessary we indicate what precisely is proved and what results this is based on.

$K$-theory is an essential ingredient in our proof of Harish-Chandra's character formula. It is possible to express the number \eqref{eq tau g index} without using $K$-theory. But then it is not clear if it equals the value of the character of a discrete series representation in the case mentioned above. This is basically because it is not clear if zero is isolated in the spectrum of the operator on $G/T$ used there. Using $K$-theory allows us to avoid this issue. Another ingredient of the proof of the fixed point formula is the equivariant coarse index.  In our setting, this is an alternative way to describe the analytic assembly map, which is convenient for our purposes.

The $K$-theoretic approach to representation theory has been studied intensively at least since the 1980s. Most of this work focused on analysing the $K$-theory group $K_*(C^*_rG)$. See the work on the Baum--Connes and Connes--Kasparov conjectures \cite{Connes94, CEN, Lafforgue02b, Wassermann87}, and Lafforgue's work on $K$-theory classes of discrete series representations \cite{Lafforgue02}. It is a challenge, however, to obtain information about representations themselves using this approach, rather than about their classes in $K$-theory. (There are positive results in this direction though,  such as  Lafforgue's independent proof in \cite{Lafforgue02} that $G$ has a discrete series if and only if it has a compact Cartan subgroup.) We hope that this paper will contribute to the understanding of the relations between representation theory, index theory and $K$-theory.

\subsection*{Acknowledgements}

The authors thank Nigel Higson, Maarten Solleveld and Yanli Song for useful discussions and advice. The second author is supported by the Australian Research Council, through Discovery Early Career Researcher Award DE160100525, and partially supported by the Simons Foundation grant 346300, and the Polish Government MNiSW 2015--2019 matching fund.


\section{Preliminaries and results}\label{sec prelim}

Throughout this paper, let $G$ be a connected semisimple Lie group. We fix a maximal compact subgroup $K<G$. The Lie algebra of a Lie group will be denoted by the corresponding lower case gothic letter. We denote complexifications by superscripts $\C$. We fix a $K$-invariant inner product on $\kg$, such as the one defined by  the Killing form and a Cartan involution. We fix Haar measures $dg$ on $G$ and $dk$ on $K$. We normalise the Haar measure $dk$ so $K$ has unit volume. (We tacitly make this choice for all compact groups.)

We consider a proper, isometric action by $G$ on a  Riemannian manifold $M$, which is cocompact (i.e.\ $M/G$ is compact). Furthermore, let $E \to M$ be a $G$-equivariant, $\Z_2$-graded, Hermitian vector bundle. Let $D$ be an odd, $G$-equivariant,  elliptic differential operator on $\Gamma^{\infty}(E)$, self-adjoint on its domain in $L^2(E)$. We will denote the restrictions of $D$ to sections of the even and odd parts of $E$ by $D^+$ and $D^-$, respectively.

The main result in this paper is Theorem \ref{thm fixed pt}, which is an expression for an equivariant index of $D$, in terms of data on fixed point sets of elements of $G$. This involves Harish-Chandra's Schwartz algebra $\cC(G)$, and traces
\[
\tau_g\colon K_0(\cC(G)) \to \C,
\]
for elements $g \in G$. These are defined in Subsection \ref{sec g trace}. Since $ K_0(\cC(G)) =  K_0(C^*_rG)$, these traces may also be viewed as traces on the $K$-theory of the reduced group $C^*$-algebra $C^*_rG$. In Subsection \ref{sec special}, we discuss some special cases of the fixed point formula, and some related results. As an application of the fixed point formula, we give a new proof of Harish-Chandra's character formula for the discrete series, Corollary \ref{cor char}.

\subsection{The $g$-trace} \label{sec g trace}

Let $\pi_0$ be the unitary representation of $G$ induced from the trivial representation of a minimal parabolic subgroup. Let $\xi$ be a unit vector in the representation space of $\pi_0$, fixed by $\pi_0(K)$. Let $\Xi$ be the matrix coefficient of $\xi$, i.e.\ for all $g \in G$,
\[
\Xi(g) = (\xi, \pi_0(g)\xi).
\]
The inner product on $\kg$ defines a $G$-invariant Riemannian metric on $G/K$. For $g \in G$, let $\sigma(g)$ be the Riemannian distance from $eK$ to $gK$ in $G/K$. For every $m \geq 0$, $X,Y \in U(\kg)$, and $f \in C^{\infty}(G)$, set
\[
\nu_{X, Y, m}(f) := \sup_{g \in G} (1+\sigma(g))^m \Xi(g)^{-1} |L(X)R(Y)f(g)|,
\]
where $L$ and $R$ denote the left and right regular representations, respectively.
\begin{definition}
The \emph{Harish-Chandra Schwartz space} $\cC(G)$ is the space of $f \in C^{\infty}(G)$ such that for all $m \geq 0$ and $X,Y \in U(\kg)$, we have $\nu_{X,Y,m}(f) < \infty$.
\end{definition}
See Section 9 in \cite{HC66}. The space $\cC(G)$ is a Fr\'echet space in the seminorms $\nu_{X,Y,m}$. It is closed under convolution, which is a continuous operation on this space (see Proposition 12.16(b) in \cite{Knapp}). Importantly, if $G$ has a discrete series, then all $K$-finite matrix coefficients of discrete series representations lie in $\cC(G)$ (see the example on page 450 in \cite{Knapp}).

Recall that an element $g \in G$ is semisimple if $\Ad(g)$ diagonalises. The set of semisimple elements contains the  open dense subset of regular elements of $G$. Let $g \in G$ be semisimple. Then its centraliser
\[
Z:= Z_G(g)
\]
is unimodular, and we have a $G$-invariant measure $d(xZ)$ on $G/Z$. Theorem 6 in \cite{HC66} states that for $m$ large enough, the orbital integral
\beq{eq int conv}
\int_{G/Z} (1+\sigma(xgx^{-1}))^{-m}\Xi(xgx^{-1})\, d(xZ)
\eeq
converges.
\begin{theorem}\label{thm orbital integral}
For all $f \in \cC(G)$, the integral
\[
\int_{G/Z}f(xgx^{-1})\, d(xZ)
\]
converges absolutely, and depends continuously on $f$.
\end{theorem}
\begin{proof}
We have for all $x \in G$, and all $m \geq 0$,
\[
|f(xgx^{-1})| \leq \nu_{0,0,m}(f) (1+\sigma(xgx^{-1}))^{-m}\Xi(xgx^{-1}),
\]
so the claim follows from convergence of the integral \eqref{eq int conv} for large enough $m$.
\end{proof}
Theorem \ref{thm orbital integral} means that $\tau_g$ is a tempered distribution on $G$ in the sense of Harish-Chandra.

\begin{definition}
For a semisimple element $g \in G$,and $f \in \cC(G)$, the \emph{$g$-trace} of $f$ is the number
\[
\tau_g(f) = \int_{G/Z}f(xgx^{-1})\, d(xZ).
\]
\end{definition}
This indeed defines a trace.
\begin{lemma}
For all $f_1, f_2 \in \cC(G)$, we have
\[
\tau_g(f_1 * f_2) = \tau_g(f_2 * f_1).
\]
\end{lemma}
\begin{proof}
For $f_1, f_2 \in C^{\infty}_c(G)$, this is a straightforward calculation involving Fubini's theorem. For general $f_1, f_2 \in \cC(G)$, this follows from continuity of $\tau_g$ and the fact that $C^{\infty}_c(G)$ is dense in $\cC(G)$ by Theorem 2 in \cite{HC66}.
\end{proof}
Since $\tau_g$ is a continuous trace on $\cC(G)$, it induces a map
\[
\tau_g \colon K_0(\cC(G)) \to \C.
\]
Note that if $g = e$, then $\tau_e$ is the von Neumann trace $f\mapsto f(e)$.

Let $C^*_rG$ be the reduced group $C^*$-algebra of $G$.
We thank Nigel Higson for pointing out the following fact to us.
\begin{theorem}\label{thm j iso}
The algebra $\cC(G)$ is contained in $C^*_rG$, and the inclusion map induces an isomorphism
\[
K_*(\cC(G)) \xrightarrow{\cong} K_*(C^*_rG).
\]
\end{theorem}
\begin{proof}
The claim  follows from the fact that  $\cC(G)$ is a dense subalgebra of Lafforgue's Schwartz algebra $\calS(G)$, and is closed under holomorphic functional calculus. Lafforgue proves in \cite{Lafforgue02b} that $\calS(G)$ is dense in $C^*_rG$ and closed under holomorphic functional calculus.
See Proposition 4.1.2 in  \cite{Lafforgue02b}, combined with Proposition  4.2.3 in \cite{Lafforgue02b}, for linear semisimple groups, and the rest of Section 4.2 in \cite{Lafforgue02b} for general reductive groups.

See also Theorem 14 and Proposition 28 in \cite{Vigneras}.
\end{proof}
Because of this fact, we obtain a map
\[
\tau_g \colon K_0(C^*_rG) \to \C.
\]


\begin{remark}\label{rem connes trace}
On page 150 of \cite{ConnesBook}, Connes mentions without proof that orbital integrals define traces on convolution algebras of groups. He appears to think of the algebra $L^1(G)$ there. 
 Using the algebra $\cC(G)$ has the advantage that it contains $K$-finite matrix coefficients of discrete series representations, which is essential to our proof of Harish-Chandra's character formula. These matrix coefficients only lie in $L^1(G)$ under certain conditions; see (1.5) in \cite{TV72} and the theorem on page 148 in \cite{HechtSchmid76}. 
\end{remark}

%
%

\subsection{The main result} \label{sec result}

The equivariant index that we will use is the analytic assembly map from the Baum--Connes conjecture \cite{Connes94}, which we denote by $\ind_G$. It takes values in $K_*(C^*_rG)$, so we obtain
\[
\ind_G(D) \in K_*(C^*_rG).
\]
If $M$ and $G$ are compact, then $K_*(C^*_rG)$ equals the representation ring of $G$, and $\ind_G(D)$ reduces to the usual equivariant index of $D$.

Let $g \in G$ be semisimple. Then we have the number
\[
\tau_g(\ind_G(D)).
\]
We will give a fixed point formula to compute such numbers.

For a point $m \in M$, we denote its stabiliser in $G$ by $G_m$.
The fixed point set $M^g$ is invariant under the centraliser $Z$ of $g$. 
\begin{lemma}\label{lem MgZ cpt}
The quotient $M^g/Z$ is compact.
\end{lemma}
\begin{proof}
Because $G$ acts properly and cocompactly, the stabiliser bundle 
\[
\Stab(M)=\{(m, g);m\in M, g\in G_m\}
\]
is $G$-cocompact under the action defined by
\[
x(m, g)=(xm, xgx^{-1}) 
\] 
for $x \in G$, $m \in M$ and $g \in G_m$.
The stabiliser bundle has the decomposition 
\[
\Stab(M)=\coprod_{(g)}\coprod_{xZ_G(g)\in G/Z_G(g)}xM^g=\coprod_{(g)}G\times_{Z_G(g)}M^g.
\]
Here $(g)$ ranges over the conjugacy classes in $G$.
Hence, $\Stab(M)/G=\coprod_{(g)}M^g/Z_G(g).$ Since $\Stab(M)/G$ is compact, so is its closed subset $M^g/Z_G(g)$.
(See page 334 of \cite{EE} for a description when $G$ is a finite group.) 
\end{proof}

For any continuous, proper action by a locally compact group $H$, with a left Haar measure $dh$, on a locally compact topological space $X$, by a \emph{cutoff function} we will mean a  continuous function $c$ on $X$ with nonnegative values, such that for all $x \in X$,
\[
\int_H c(hx)\, dh = 1.
\]
These always exist, and can be chosen to be compactly supported if $X/H$ is compact. So by Lemma \ref{lem MgZ cpt}, there is a compactly supported cutoff function $c^g \in C_c(M^g)$ for the action by $Z$ on $M^g$.

Let $\cN \to M^g$ be the normal bundle to $M^g$ in $M$. The connected components of $M^g$ are submanifolds of $M$ of possibly different dimensions, so the rank of $\cN$ may jump between these components. In what follows, we implicitly apply all constructions to the connected components of $M^g$ and add the results together. Suppose that $g$ is contained in a compact subgroup of $G$. Then the closure of the set of its powers is a torus $T^g <G$. This torus acts trivially on $M^g$. Consider the class
\[
[\Bigwedge \cN \otimes \C] \in K^0_{T^g}(\supp (c^g)) = K^0(\supp (c^g)) \otimes R(T^g)
\]
in the equivariant topological $K$-theory of $\supp (c^g)$. (We use the fact that $\supp (c^g) \subset M^g$ is compact.) Here $R(T^g)$ is the representation ring of $T^g$, which we view as the ring of characters, and $\Bigwedge \cN \otimes \C$ is graded according to parities of exterior powers. Evaluating characters at $g$, applied to   the factor in $R(T^g)$, yields the class
\[
[\Bigwedge \cN \otimes \C] (g) \in K^0(\supp (c^g)) \otimes \C.
\]
This class is invertible with respect to the ring structure defined by tensor products, see Lemma 2.7 in \cite{Atiyah68}.

Let $\sigma_D$ be the principal symbol of $D$. It defines a class
\[
[\sigma_D|_{\supp(c^g)}] \in K^0_{T^g}(TM^g|_{\supp(c^g)}) = K^0(TM^g|_{\supp(c^g)}) \otimes R({T^g}).
\]
Again, we evaluate at $g$ to obtain 
\[
[\sigma_D|_{\supp(c^g)}](g) \in K^0(TM^g|_{\supp(c^g)}) \otimes \C.
\]
Consider the Chern characters
\[
\begin{split}
\ch\colon & K^0(\supp(c^g)) \to H^*(\supp(c^g));\\
\ch\colon & K^0(TM^g|_{\supp(c^g)}) \to H^*(TM^g|_{\supp(c^g)}),
\end{split}
\]
and the Todd class
\[
\Todd(TM^g \otimes \C) \in H^*(TM^g).
\]
The cohomology group $H^*(\supp(c^g))$ acts on $H^*(TM^g|_{\supp(c^g)})$ via pullback along the tangent bundle projection.

Our main result is the following fixed point formula.
\begin{theorem}\label{thm fixed pt}
If $G/K$ is odd-dimensional, then for all semisimple $g \in G$,
\[
\tau_g(\ind_G(D)) = 0.
\] 
If $G/K$ is even-dimensional, then
for almost all semisimple $g \in G$ (see Remark \ref{rem cond g}), we have
\[
\tau_g(\ind_G(D)) = 0
\]
if $g$ is not contained in a compact subgroup of $G$, and 
\beq{eq fixed pt}
\tau_g(\ind_G(D)) = \int_{TM^g}c^g \frac{\ch\bigl([\sigma_D|_{\supp(c^g)}](g)\bigr)\Todd(TM^g \otimes \C)}{\ch\bigl([\Bigwedge \cN \otimes \C] (g) \bigr)}
\eeq
if it is.
\end{theorem}

\begin{remark} \label{rem cond g}
The fixed point formula in Theorem \ref{thm fixed pt} holds for almost all semisimple $g \in G$ contained in compact subgroups. In fact, we can be more specific about the condition on $g$ for the formula to hold. Let $d$ be the Riemannian distance on $G$ corresponding to the left invariant Riemannian metric defined by the inner product on $\kg$. Consider the function $\psi \in C(G)$ mapping $x \in G$ to 
\beq{eq def psi}
\psi(x) = e^{-d(e,x)^2}.
\eeq
The condition on $g$ is that the integral
\[
\int_{G/Z}\psi(xgx^{-1})\, d(xZ)
\]
converges. (We then say that $g$ has \emph{finite Gaussian orbital integral}, see Definition \ref{def FGOI}.) We will show in Proposition \ref{prop FGOI ae} that this is true for almost all $g \in G$. But more specifically, this condition holds for example if $G/Z$ is compact, so in particular if $g = e$ or if $G$ is compact. By Theorem \ref{thm orbital integral}, the condition holds for all semisimple $g$ if $\psi \in \cC(G)$. 
We will see in Lemma \ref{lem varphi L1} that $\psi \in L^1(G)$; therefore, if the integral defining $\tau_g(f)$ converges for all (continuous) $f \in L^1(G)$ and all semisimple $g \in G$, then Theorem \ref{thm fixed pt} holds for all semisimple $g \in G$. (This seems to be implied on page 150 of \cite{ConnesBook}.)

For our application of Theorem \ref{thm fixed pt} to Harish-Chandra's character formula, it is enough for the formula to hold almost everywhere. More generally, the distributional index discussed in Subsection \ref{sec dist} is completely determined by Theorem \ref{thm fixed pt}.
\end{remark}

\subsection{Special cases and related results} \label{sec special}

If $M$ and $G$ are compact, then we may take $c^g$ to be constant $1$. Then Theorem \ref{thm fixed pt} reduces to the Atiyah--Segal--Singer fixed point formula, see Theorem 2.12 in \cite{Atiyah68} and (3.1) in \cite{ASIII}. In fact, in this compact case, the proof of Theorem \ref{thm fixed pt} applies directly without the assumption that $G$ is semisimple, and for all $g \in G$.

If we take $g = e$, then the fixed point formula (which is then just an index theorem, since $M^e = M$) holds by Remark \ref{rem cond g}. In that case,  Theorem \ref{thm fixed pt} reduces to the case of the main result of \cite{Wang14}, Proposition 6.11 in that paper, for connected semisimple Lie groups. In this case $\tau_e(\ind_G(D))$ is precisely Atiyah's $L^2$-index of $D$, defined for discrete groups in \cite{AtiyahL2}. (See Proposition 4.4 in \cite{Wang14}.) Proposition 6.11 in \cite{Wang14} generalises Atiyah's $L^2$-index theorem, Theorem 3.8 in \cite{AtiyahL2}, from discrete to arbitrary Lie groups. If $M = G/H$, for a compact subgroup $H<G$, then the case of Proposition 6.11 in \cite{Wang14} implied by Theorem \ref{thm fixed pt} in turn reduces to the case of Connes and Moscivici's index theorem, Theorem 5.1 in \cite{Connes82}, for connected semisimple groups. (See Corollary 6.14 and Remark 6.15 in \cite{Wang14}.)

In \cite{Wangwang}, a version of Theorem \ref{thm fixed pt} for actions by discrete groups on orbifolds is proved, see Theorem 6.1 in that paper. In a sense, this result is orthogonal to the case of connected groups we consider here, and requires a completely different set of techniques.

If $M^g$ is compact, then the right hand side of \eqref{eq fixed pt} equals the right hand side of (2.8) in \cite{HW}. Hence the $g$-index of $D$, as defined in \cite{HW}, equals $\tau_g(\ind_G(D))$. 

In this paper, the special case of Theorem \ref{thm fixed pt} we are most interested in is Harish-Chandra's character formula for the discrete series, as we will discuss next.

\subsection{Harish-Chandra's character formula} \label{sec char form}

Suppose that $\rank(G) = \rank(K)$, so that $G$ has  discrete series  representations. Let $T<K$ be a maximal torus. Let $\pi$ be a discrete series representation of $G$, with Harish-Chandra parameter $\lambda \in i\kt^*$. Let $R^+$ be the set of roots of $(\kg^{\C}, \kt^{\C})$ with positive inner products with $\lambda$. Let $\rho$ be half the sum of the elements of $R^+$. Let $W_c := N_K(T)/T$ be the Weyl group of $(K,T)$.

Let $\Theta_{\pi} \in \calD'(G)$ be the global character of $\pi$. On the regular elements of $G$, it is given by an analytic function, which we also denote by $\Theta_{\pi}$. Harish-Chandra's character formula, Theorem 16 in \cite{HC66}, is a special case of Theorem \ref{thm fixed pt}.
\begin{corollary} \label{cor char}
For all regular elements $g \in T$, we have
\beq{eq char}
\Theta_{\pi}(g) = (-1)^{\dim(G/K)/2}\frac{\sum_{w\in W_c} \sign(w)e^{w\lambda}}{e^{\rho}\prod_{\alpha \in R^+} (1-e^{- \alpha})}(g).
\eeq
\end{corollary} 

\begin{remark}
Because the character formula is an integral part of Harish--Chandra's classification of the discrete series in \cite{HCI, HC66}, it is worth specifying what exactly the statement is  in Corollary \ref{cor char}. Let $\lambda \in i\kt^*$ be regular, 
and suppose that $\lambda + \rho$ is integral.
 The condition that $\rank(G) = \rank(K)$ implies that $G$ has an irreducible unitary representation $\pi$ with square integrable matrix coefficients and infinitesimal character $\chi_{\lambda}\colon Z(U(\kg^{\C})) \to \C$ corresponding to $\lambda$ via the Harish-Chandra homomorphism, and whose lowest $K$-type has highest weight $\lambda + \rho - 2\rho_c$. Here $\rho_c$ is half the sum of the compact roots in the positive system $R^+$ determined by $\lambda$. (See for example Theorem 9.20 in \cite{Knapp}.) Corollary \ref{cor char} states that the global character of this representation is given by \eqref{eq char} on $T$. (Compare this with for example  Theorem 12.7(a) in \cite{Knapp}.) This information can then be used to prove that $G$ can only have discrete series if $\rank(G) = \rank(K)$, and that the characters obtained in this way exhaust the discrete series. But we will not use those facts.

In Remark \ref{rem indep}, we make further comments on the independence of Corollary \ref{cor char} of existing results.
\end{remark}

\subsection{A distributional index}\label{sec dist}

Instead of considering the numbers $\tau_g(\ind_G(D))$ for individual elements $g \in G$, we can assemble them into a distributional index of $D$. Let $G^{\reg}\subset G$ be the subset of regular elements.
\begin{lemma}
For all $\varphi \in C^{\infty}_c(G^{\reg})$, and $f \in \cC(G)$, the integral
\[
\int_{G^{\reg}} \varphi(g) \tau_g(f)\, dg
\]
converges, and depends continuously on $\varphi$ and $f$. (Here we use the usual Fr\'echet topology on $C^{\infty}_c(G)$.)
\end{lemma}
\begin{proof}
On each connected component of $G^{\reg}$, all centralisers are conjugate to a fixed Cartan subgroup of $G$. This implies that
the function $g \mapsto \tau_g(f)$ on $G^{\reg}$ is continuous. 
%
\end{proof}

By this lemma, we obtain a continuous trace
\[
\tau\colon \cC(G) \to \calD'(G^{\reg}),
\]
defined by 
\[
\langle \tau(f), \varphi \rangle = \int_G \varphi(g) \tau_g(f)\, dg,
\]
for  $f \in \cC(G)$ and $\varphi \in C^{\infty}_c(G)$. It induces
\[
\tau\colon  K_0(\cC(G)) \to \calD'(G^{\reg}).
\]
\begin{definition} \label{def dist index}
The \emph{distributional index} of $D$ is
\[
\tau(\ind_G(D)) \quad \in \calD'(G^{\reg}).
\]
\end{definition}
Theorem \ref{thm fixed pt} determines this distribution completely.

It is an interesting question if, for $f \in \cC(G)$, the distribution $\tau(f)$ extends to all of $G$. Then the index in Definition \ref{def dist index} would extend to an index with values in $\calD'(G)$, which is determined by Theorem \ref{thm fixed pt}. We will see in Section \ref{sec char}
that the distributional index of the Dirac operator used by Schmid \cite{Schmid76} to realise discrete series representations is the character of such a representation.


\section{Another trace}

Fix an element $g \in G$. An important tool in the proof of Theorem \ref{thm fixed pt} is a trace $\Tr_g$ defined on certain operators on $L^2(E)$.
%
%
This trace is defined in terms of Schwartz kernels of operators. It has the advantage that it can be evaluated in terms of analysis and geometry, whereas $\tau_g$ has the advantage that it can be used to compute values of characters of representations.

In this section, we use the trace $\Tr_g$ to give an expression for the number $\tau_g(\ind_G(D))$ in terms of the heat kernel associated to $D$. (See Proposition \ref{prop taug index}.) In Section \ref{sec loc}, we localise that expression to give a proof of Theorem \ref{thm fixed pt}.

We will often use the differentiable version of Abels' slice theorem, see page 2 of \cite{Abels}. This states that there is a $K$-invariant submanifold $N\subset M$, such that the action map $[g, n] \mapsto gn$, for $g \in G$ and $n \in N$, defines a $G$-equivariant diffeomorphism
\[
G\times_K N \xrightarrow{\cong} M.
\]
Here $G\times_K N$ is the quotient of $G \times N$ by the action by $K$ given by
\[
k\cdot (g,n) = (gk^{-1}, kn)
\]
for $k \in K$, $g \in G$ and $n \in N$. We fix such a submanifold $N$ from now on.

\subsection{Schwartz kernels}

The $G$-equivariant vector bundle $E \to M$ decomposes as
\[
E \cong G\times_K(E|_N) \to G\times_K N = M.
\]
Hence
\[
\Gamma^{\infty}(E) \cong \bigl(C^{\infty}(G) \hat \otimes \Gamma^{\infty}(E|_N) \bigr)^K,
\]
where $\hat \otimes$ denotes the completion of the algebraic tensor product in the Fr\'echet topology on $\Gamma^{\infty}(E)$ (which is well-defined since $\Gamma^{\infty}(\Hom(E|_N))$ is nuclear), and the superscript $K$ denotes the subspace of $K$-invariant elements.

Consider the vector bundle
\[
\Hom(E)\to M \times M
\]
with fibres
\[
\Hom(E)_{(m, m')} = \Hom(E_m, E_{m'}),
\]
for $m, m' \in M$. The bundle $\Hom(E|_N)\to N \times N$ is defined analogously.
Consider the action by $K \times K$ on the space
\[
\cC(G) \otimes \Gamma^{\infty}(\Hom(E|_N))
\]
given by
\[
(k,k')\cdot(f\otimes A) (g, n, n')= f(kgk'^{-1})k^{-1}A(kn, k'n')k',
\]
for $k, k' \in K$, $f \in \cC(G)$, $A \in \Gamma^{\infty}(\Hom(E|_N))$, $g \in G$ and $n, n' \in N$. 
We have the space
\[
\tilde \cC(E) := \bigl( \cC(G) \hat \otimes \Gamma^{\infty}(\Hom(E|_N)) \bigr)^{K\times K},
\]
where again,  $\hat \otimes$ denotes the completion of the algebraic tensor product in the tensor product Fr\'echet topology. This space is a Fr\'echet algebra with respect to convolution.

For $\tilde \kappa \in \tilde \cC(E)$, consider the operator $T_{\tilde \kappa}$ on $L^2(E)$ defined by
\[
(T_{\tilde \kappa}s)(gn) = \int_G \int_N g\tilde \kappa(g^{-1}g', n, n')g'^{-1} s(g'n')\, dn'\, dg',
\]
for $s \in L^2(E)$, $g \in G$ and $n \in N$. We will see in Lemma \ref{lem CE C*MG} that this defines a bounded operator $T_{\tilde \kappa}$ on $L^2(E)$, although at this point it is not very important to us on which space $T_{\tilde \kappa}$ acts. This operator is $G$-equivariant, and has Schwartz kernel $\kappa \in \Gamma^{\infty}(\Hom(E))^G$ given by
\beq{eq kappa tilde}
\kappa(gn, g'n') = g \tilde \kappa(g^{-1}g', n, n') g'^{-1},
\eeq
for $g, g' \in G$ and $n, n' \in N$. Given $\tilde \kappa \in \tilde \cC(E)$ we will always write $\kappa$ for the section in $\Gamma^{\infty}(\Hom(E))^G$ defined like this.

The assignment $\tilde  \kappa \mapsto T_{\tilde \kappa}$ is injective, and satisfies
\[
T_{\tilde \kappa} \circ T_{\tilde \kappa'} = T_{\tilde \kappa * \tilde \kappa'},
\]
for $\tilde \kappa, \tilde \kappa' \in \tilde \cC(E)$. 
\begin{definition}\label{def SE}
The algebra 
 $\cC(E)$ is defined as
 \[
 \cC(E) = \{T_{\tilde \kappa}; \tilde \kappa \in \tilde \cC(E)\}. 
 \]
 The topology on this algebra is the one corresponding to the topology on $\tilde \cC(E)$ via the isomorphism $\tilde \kappa \mapsto T_{\tilde \kappa}$.
\end{definition}
Note that $\cC(E)$ is a Fr\'echet algebra because $\tilde \cC(E)$ is.

It will be important to us that heat kernels corresponding to twisted $\Spinc$-Dirac operators lie in $\cC(E)$, see Lemma \ref{lem heat SE}.

\subsection{The trace $\Tr_g$}

Let $c \in C^{\infty}_c(M)$ be a cutoff function for the action by $G$ on $M$. Let $g \in G$ be semisimple. We will denote the fibre-wise trace of endomorphisms of $E$ by $\tr$. 
\begin{lemma} \label{lem Trg conv}
Let $\tilde \kappa \in \tilde \cC(E)$.
\begin{enumerate}
\item[(a)] The integral
\beq{eq Trg conv}
\int_{G/Z} \int_M c(xgx^{-1}m)\tr(xgx^{-1} \kappa(xg^{-1}x^{-1}m, m))\, dm\, d(xZ)
\eeq
converges absolutely, and depends continuously on $\tilde \kappa$. 
\item[(b)] 
Let $c_G$ be a cutoff function for the action by $Z$ on $G$ by right multiplication, and let $c^g$ be the function on $M$ defined by
\[
c^g(m) = \int_G c_G(x)c(xgm)\, dx,
\]
for $m \in M$.
The integral \eqref{eq Trg conv} equals
\beq{eq int tilde c}
\int_M c^g(m) \tr(\kappa(m, gm) g)\, dm.
\eeq
\item[(c)]
The integral \eqref{eq Trg conv} equals
\beq{eq int N GZ}
\int_N \int_{G/Z} \tr( \tilde \kappa(xgx^{-1}, n, n))\, d(xZ) \, dn,
\eeq
so in particular it is independent of the cutoff function $c$.
\end{enumerate}
\end{lemma}
\begin{proof}
If $m \in M$ and $x \in G$, and $m' = x^{-1}m$, then the equivariance property of $\kappa$ implies that
\[
\begin{split}
c(xgx^{-1}m)\tr(xgx^{-1} \kappa(xg^{-1}x^{-1}m, m) &= c(xgm')\tr(xgx^{-1} \kappa(xg^{-1}m', xm') \\
&= c(xgm')\tr( \kappa(m', gm') g).
\end{split}
\]
So
\begin{multline}\label{eq int 2}
\int_{G/Z} \int_M c(xgx^{-1}m)\tr(xgx^{-1} \kappa(xg^{-1}x^{-1}m, m))\, dm\, d(xZ)\\
 = \int_G c_G(x)\int_M  c(xgm')\tr( \kappa(m', gm') g) \, dm' \, dx. 
 \end{multline}
 Now for all $y \in G$ and $n \in N$ we have
 \beq{eq traces}
 \tr( \kappa(yn, gyn) g) = \tr( y\tilde \kappa(y^{-1}gy, n, n)y^{-1}) =  \tr( \tilde \kappa(y^{-1}gy, n, n)). 
 \eeq
 So  the right hand side of \eqref{eq int 2} equals
 \[
 \int_G c_G(x)  \int_N \int_G  c(xgyn)\tr( \tilde \kappa(y^{-1}gy, n, n)) \,  dy \, dn \, dx. 
 \]
 Substituting $y' = xgy$ for $y$, we find that this equals
 \begin{multline*}
  \int_G c_G(x) \int_N \int_G   c(y'n)\tr( \tilde \kappa(y'^{-1}xgx^{-1}y', n, n)) \,  dy' \, dn\, dx \\
  \leq   \int_G c_G(x)\int_N \int_G  c(y'n)|\tr( \tilde \kappa(y'^{-1}xgx^{-1}y', n, n))| \,  dy' \, dn\, dx
  \end{multline*}
The integrand on the right hand side is nonnegative, so by Fubini's theorem, the integral equals  
\beq{eq int 3}
  \int_N  \int_G   c(y'n) \int_G c_G(x) |\tr( \tilde \kappa(y'^{-1}xgx^{-1}y', n, n))| \, dx  \,  dy' \, dn. 
\eeq
 Since $c_G$ is a cutoff function for right multiplication by $Z$ on $G$, we have for all $y' \in G$ and $n \in N$,
 \begin{multline*}
 \int_G c_G(x)| \tr( \tilde \kappa(y'^{-1}xgx^{-1}y', n, n))| \, dx \\
 = \int_{G/Z} | \tr( \tilde \kappa(y'^{-1}xgx^{-1}y', n, n))| \, d(xZ) \\
  = \int_{G/Z} | \tr( \tilde \kappa(xgx^{-1}, n, n))| \, d(xZ)
 \end{multline*}
 which converges for all $y' \in G$ and $n \in N$ by Theorem \ref{thm orbital integral}. So \eqref{eq int 3} equals
 \begin{multline*}
\int_N \int_G   c(y'n)\, dy'  \int_{G/Z} | \tr( \tilde \kappa(xgx^{-1}, n, n))| \, d(xZ) \, dn  \\
= \int_N \int_{G/Z} | \tr( \tilde \kappa(xgx^{-1}, n, n))| \, d(xZ) \, dn \\
\leq \vol(N) \int_{G/Z} | \tr( \tilde \kappa(xgx^{-1}, n, n))| \, d(xZ).
 \end{multline*}
 We conclude that the integral \eqref{eq Trg conv} converges absolutely. It depends continuously on $\tilde \kappa$ by the same argument as in the proof of Theorem \ref{thm orbital integral}. Because the integral converges absolutely, we may switch the order of integration on the right hand side of \eqref{eq int 2} to conclude that \eqref{eq Trg conv} equals \eqref{eq int tilde c}. Furthermore, we may omit absolute value signs in the above calculations to find that \eqref{eq Trg conv} equals \eqref{eq int N GZ}.
%
%
%
\end{proof}

\begin{definition} \label{def Trg}
Let $\tilde \kappa \in \tilde \cC(E)$.
The \emph{$g$-trace} of the operator
 $T = T_{\tilde \kappa} \in \cC(E)$ is
\[
\Tr_g(T) = \int_{G/Z} \int_M c(xgx^{-1}m)\tr(xgx^{-1} \kappa(xg^{-1}x^{-1})m, m))\, dm\, d(xZ).
\]
\end{definition}


\begin{lemma} \label{lem Trg trace}
For all $S, T \in \cC(E)$, we have
\[
\Tr_g(ST) = \Tr_g(TS).
\]
\end{lemma}
\begin{proof}
This proof is analogous to the proof of Proposition 3.18 in \cite{Wangwang}.

Let $\tilde \kappa_S, \tilde \kappa_T \in \tilde \cC(E)$ be the kernels defining $S$ and $T$, respectively. Define the function $\mu$ on $M \times M$ by
\[
\mu(m, m') = \int_{G/Z} \tr(\kappa_T(m, m')\kappa_S(m', xgx^{-1}m)xgx^{-1})\, d(xZ)
\]
for $m, m' \in M$.
Then for all $x \in G$ and $m,m' \in M$, we have $\mu(xm, xm') = \mu(m, m')$. Part (2) of Lemma 3.10 in \cite{Wangwang} states that property implies that
\[
\int_M \int_M c(m) \mu(m, m') \, dm, dm' = \int_M \int_M c(m') \mu(m, m') \, dm, dm'. 
\]
Now, using Fubini's theorem and the definition of $\Tr_g$, one can show directly that the left hand side of the above equality equals $\Tr_g(TS)$, whereas the right hand side equals $\Tr_g(ST)$.
\end{proof}


\subsection{The coarse index}

We will first prove Theorem \ref{thm fixed pt} for twisted $\Spinc$-Dirac operators. Therefore, we now suppose that $M$ is even-dimensional,  has a $G$-equivariant $\Spinc$-structure with spinor bundle $S\to M$, and that $E = S \otimes W$, for a $G$-equivariant Hermitian vector bundle $W \to M$. We suppose $D$ is a twisted $\Spinc$-Dirac operator on $E$. 

To give an expression for $\tau_g(\ind_G(D))$, we will use an alternative definition of the analytic assembly map, in terms of the \emph{coarse index}. The coarse index takes values in the $K$-theory of the \emph{(reduced) equivariant Roe algebra} $C^*(M)^G$ of $M$, which may be realised as follows. A section $\kappa$ of $\Hom(E)$ is said to have \emph{finite propagation} if there is an $R>0$ such that for all $m, m' \in M$ with $d(m, m') \geq R$, we have $\kappa(m, m') = 0$. We will consider locally integrable sections $\kappa$ of $\Hom(E)$ for which the expression
\[
(T_{\kappa}s)(m) = \int_M \kappa(m, m')s(m')\, dm',
\]
for $s \in L^2(E)$, defines a bounded operator  $T_{\kappa}\in \cB(L^2(E))$. Then $T_{\kappa}$ is $G$-equivariant if $\kappa$ is $G$-invariant, in the sense that for all $m, m' \in M$ and $g \in G$,
\[
\kappa(gm, gm') = g\kappa(m, m')g^{-1}.
\]
\begin{definition}
The \emph{equivariant Roe algebra} $C^*(M)^G$ of $M$ is the closure in $\cB(L^2(E))$ of the algebra of $G$-equivariant, bounded operators of the form $T_{\kappa}$, where $\kappa$ is a locally integrable section of $\Hom(E)$ with finite propagation.
\end{definition}

\begin{lemma}\label{lem CE C*MG}
The algebra $\cC(E)$ is a dense subalgebra of $C^*(M)^G$.
\end{lemma}
\begin{proof}
Note that
\beq{eq decomp L2E}
L^2(E) = \bigl(L^2(G) \hat \otimes L^2(E|_N) \bigr)^K
\eeq
(where the hat denotes the completion in the $L^2$-norm). The algebras $\cC(E)$ and $C^*(M)^G$ have the joint dense subalgebra $\cC_c(E)$ of operators with kernels in 
\[
(C^{\infty}_c(G) \otimes \Gamma^{\infty}(\Hom(E|_N)))^{K\times K}.
\]
If $\tilde \kappa$ lies in this space, then the corresponding operator on 
 \eqref{eq decomp L2E} acts on the factor $L^2(G)$ by (left) convolution by the factor of $\tilde \kappa$ in $C^{\infty}_c(G)$. Since the completion of the algebra of convolution operators on $L^2(G)$ by functions in $C^{\infty}_c(G)$ is $C^*_rG$, and $C^*_rG$ contains $\cC(G)$, we conclude that indeed $\cC(E) \subset C^*(M)^G$. We have also seen that this  subalgebra is dense.
\end{proof}

The coarse index of $D$, denoted by $\ind_{C^*(M)^G}(D)$, is the element of $K_0(C^*(M)^G)$ explicitly given by
\begin{multline}\label{eq coarse index}
\ind_{C^*(M)^G}(D) =\\ 
\left[ \mattwo{e^{-tD^-D^+}}{e^{-\frac{t}{2} D^-D^+} \frac{1-e^{-tD^-D^+}}{D^-D^+}D^-}{e^{-\frac{t}{2} D^+D^-} \frac{1-e^{-tD^+D^-}}{D^+D^-}D^+}{1-e^{-tD^+D^-}}\right] 
- \left[ \mattwo{0}{0}{0}{1}\right],
\end{multline}
for any $t>0$.
(See Exercise 12.7.3 in \cite{Higson00}, which can be solved as on page 356 of \cite{CM90}.)

Since $M/G$ is compact, the Roe algebra $C^*(M)^G$ is Morita equivalent to $C^*_rG$. The corresponding
 isomorphism
\beq{eq iso Roe}
K_*(C^*(M)^G) \cong K_*(C^*_rG)
\eeq
can be described as follows.
Consider
the map $\Tr_N\colon  \tilde \cC(E)\to \cC(G)$ given by
\beq{eq def TrN}
\Tr_N({\tilde \kappa})(x) = \int_{N} \tr(\tilde \kappa(x, n, n))\, dn,
\eeq
for $x \in G$. We will also write $\Tr_N(T_{\tilde \kappa}) := \Tr_N(\tilde \kappa)$.  Then for $f \in \cC(G)$ and $ \kappa_N \in \Gamma^{\infty}(\Hom(E|_N))$ such that $f \otimes  \kappa_N \in \cC(E)$, we have $\Tr_N(f\otimes \kappa_N) = f \Tr(T_{\kappa_N})$, where $\Tr$ is the operator trace, and $T_{\kappa_N}$ is  the trace class operator on $L^2(E|_N)$ with smooth kernel $\kappa_N$.

Let $p \in M_{\infty}(\cC(G))$ and $q \in M_{\infty}(\Gamma^{\infty}(\Hom(E|_N)))$ be projections such that $p \otimes q \in M_{\infty}(\tilde \cC(E))$.  Then
\[
\Tr_N(p \otimes q) = p \otimes \Tr(T_q) \quad \in M_{\infty}(\cC(G))\otimes M_{\infty}(\C),
\]
where on both sides, traces and the construction $\tilde \kappa_N \mapsto T_{\tilde \kappa_N}$ are applied entry-wise. Let $\tr$ be the matrix trace on $M_{\infty}(\C)$. Since $\tr(\Tr(T_q))$ is an integer, we obtain a projection
\beq{eq TrN pq}
(1\otimes \tr)(\Tr_N(p \otimes q) ) =\tr(\Tr(T_q)) p   \in M_{\infty}(\cC(G)).
\eeq
\begin{lemma}\label{lem iso Kthry}
The extension of the above construction to  the dense subspace of $C^*(M)^G$ on which $\Tr_N$ is well-defined induces the isomorphism \eqref{eq iso Roe}.
\end{lemma}
\begin{proof}
Let the Hilbert $C^*_rG$-module  $\cE$ be the completion of $\Gamma_c(E)$ in the $C^*_rG$-valued inner product given by
\[
(s, s')_{C^*_rG}(x) = (s, x\cdot s')_{L^2(E)},
\]
for $s, s' \in \Gamma_c(E)$ and $x \in G$. Then 
\[
L^2(E) \cong \cE \otimes_{C^*_rG} L^2(G),
\]
and the map $T \mapsto T \otimes 1$ defines an isomorphism $\calL(\cE) \xrightarrow{\cong} \cB(L^2(E))$. (Here $\calL(\cE)$ is the algebra of adjointable operators on $\cE$.) See Lemma 2.2 in \cite{Roe02}. This isomorphism restricts to an isomorphism
\[
\cK(\cE) \xrightarrow{\cong} C^*(M)^G,
\]
see Lemma 2.3 in \cite{Roe02}. Now $\cK(\cE)$ is Morita equivalent to $C^*_rG$, and this is how the isomorphism \eqref{eq iso Roe} comes about.

The isomorphism
\beq{eq iso E C*G}
K_*(\cK(\cE)) \xrightarrow{\cong} K_*(C^*_rG)
\eeq
induced by Morita equivalence is induced by an operator trace, analogously to the isomorphism $K_*(\cK(L^2(E))) \cong \Z$ when $G$ is trivial.
Let us make this more explicit. The isomorphism
\[
\Gamma_c(E) \cong (C_c(G) \otimes \Gamma(E|_N))^K
\]
extends continuously to an embedding
\[
\cE \hookrightarrow C^*_rG \otimes L^2(E|_N).
\]
This induces an injective $*$-homomorphism
\beq{eq incl cE}
\cK(\cE) \hookrightarrow C^*_rG \otimes \cK(L^2(E|_N)).
\eeq
The isomorphism
\beq{eq iso tr}
K_*(C^*_rG \otimes \cK(L^2(E|_N))) \xrightarrow{\cong} K_*(C^*_rG)
\eeq
induced by Morita equivalence is induced by the operator trace on $\calL^1(L^2(E|_N)) \subset \cK(L^2(E|_N))$, applied to projections as above the lemma. (Here $\calL^1(L^2(E|_N))$ is the algebra of trace-class operators on $L^2(E|_N)$.) We have seen that this operator trace corresponds to the map $\Tr_N$ on kernels. The isomorphism \eqref{eq iso E C*G}
is the composition of the map induced by the embedding \eqref{eq incl cE} and the isomorphism \eqref{eq iso tr}. So the claim follows.
\end{proof}

 

Roe showed in \cite{Roe02} that the isomorphism \eqref{eq iso Roe} maps the coarse index to the analytic assembly map. The reason why we use this description of the analytic assembly map is that it does not require us to use properly supported operators. Since the heat operators $e^{-tD^-D^+}$ and $e^{-tD^+D^-}$ are not properly supported, using the more standard definition of the assembly map as in \cite{Connes94} would lead to technical issues.

\subsection{Computing $\tau_g(\ind_G(D))$}

\begin{lemma} \label{lem heat SE}
For all $t>0$, the heat operators $e^{-tD^-D^+}$ and $e^{-tD^+D^-}$ lie in $\cC(E)$.
\end{lemma}
This fact will be proved in Subsection \ref{sec decomp kt}. Let $g \in G$ be semisimple. Then Lemma \ref{lem heat SE} implies that the traces $\Tr_g(e^{-tD^-D^+})$ and $\Tr_g(e^{-tD^+D^-})$ are well-defined.
\begin{proposition}\label{prop taug index Trg}
For all $t>0$, we have
\[
\tau_g(\ind_G(D)) = \Tr_g(e^{-tD^-D^+}) - \Tr_g(e^{-tD^+D^-}).
\]
\end{proposition}
\begin{proof}
The map $\Tr_N$ was implicitly defined to map the unit $1$ added to $C^*(M)^G$ to the unit $1$ added to $C^*_rG$. The matrix elements in \eqref{eq coarse index}, apart from the terms $1$, are smooth kernel operators. Therefore, they lie in the domain of $\Tr_N$. 
So, by Lemma \ref{lem iso Kthry},
\begin{multline*}
\ind_G(D) = 
\Tr_N(\ind_{C^*(M)^G}(D) )=\\ 
[\Tr_N(e^{-tD^-D^+}) + 1- \Tr_N(e^{-tD^+D^-})] - [1]
\\\in K_0(C^*_rG).
\end{multline*}
Note that the $2\times 2$ matrices appearing in \eqref{eq coarse index} are in fact single bounded operators on $L^2(E)$, i.e. $1\times 1$ matrices over $\cB(L^2(E))$; they only appear as $2\times 2$ matrices because of the grading of $E$. Therefore, the map $1\otimes \tr$ in \eqref{eq TrN pq} does not appear in this case. There is a sum over the diagonals of these $2\times 2$ matrices because of the fibrewise trace of endomorphisms in \eqref{eq def TrN}.

Now by Lemma \ref{lem heat SE}, we have  $\Tr_N(e^{-tD^-D^+}) \in \cC(G)$ and $\Tr_N(e^{-tD^+D^-}) \in \cC(G)$. Since the extension of a trace to the unitisation of an algebra is by definition equal to zero on the added unit, we obtain
\[
\tau_g (\ind_G(D)) = \tau_g \circ \Tr_N (e^{-tD^-D^+}) - \tau_g \circ \Tr_N (e^{-tD^+D^-}).
\]
Part (c) of Lemma \ref{lem Trg conv} states that $\tau_g \circ \Tr_N = \Tr_g$, so the claim follows.
\end{proof}

Let $\Str$ be the fibre-wise supertrace on endomorphisms of the $\Z_2$-graded vector bundle $E$. Combining part (b) of Lemma \ref{lem Trg conv} with  Proposition \ref{prop taug index Trg}, we reach the main conclusion of this section.
\begin{proposition}\label{prop taug index}
Let $\kappa_t$ be the Schwartz kernel of $e^{-tD^2}$. Let $c^g$ be as in part (b) of Lemma \ref{lem Trg conv}. 
Then for all $t > 0$,
\[
\tau_g(\ind_G(D)) =  \int_M c^g(m) \Str(\kappa_t(m, gm) g)\, dm.
\]
\end{proposition}
This result involves a supertrace because of the difference on the right hand side of the expression in Proposition \ref{prop taug index Trg}: the operators $e^{-tD^-D^+}$ and $e^{-tD^+D^-}$ act on even and odd degree sections of $E$, respectively.

\section{Localisation}\label{sec loc}

In this section, we use heat kernel localisation techniques to prove Theorem \ref{thm fixed pt}. The central step is an estimate for heat kernels, Proposition \ref{prop est kM}. This implies the fixed point formula for twisted $\Spinc$-Dirac operators, which generalises to elliptic operators in the usual way.

As in the previous section, we suppose that $D$ is a twisted $\Spinc$-Dirac operator, except where stated otherwise (in the proof of Theorem \ref{thm fixed pt} in Subsection \ref{sec proof}).

\subsection{Decomposing heat kernels} \label{sec decomp kt}

Suppose that $G/K$ is even-dimensional.
Let $\kg = \kk \oplus \kp$ be a Cartan decomposition. We may assume that the adjoint representation $\Ad\colon K\to \SO(\kp)$ lifts to $\widetilde{\Ad}\colon K \to \Spin(\kp)$. Indeed, this is true for a double cover $\tilde G$ of $G$, whose maximal compact subgroup $\tilde K$ is a double cover of $K$. 
Then $\tilde G$ acts properly on $M$ via the covering map $\tilde G \to G$, so that
\[
M = \tilde G \times_{\tilde K} \tilde N = G\times_K N,
\]
for $K$-invariant submanifolds $\tilde N, N \subset M$. So if necessary, we replace $G$ by $\tilde G$, so that the map $\widetilde{\Ad}$ exists.
Let $S_{\kp}$ be the standard representation of $\Spin(\kp)$. It is $\Z_2$-graded since $\kp$ is even-dimensional. We view it as a representation of $K$ via the map
\[
K \xrightarrow{\widetilde{\Ad}} \Spin(\kp) \to \End(S_{\kp}).
\]

By Proposition 3.10 in \cite{HM14}, the slice $N$ has a $K$-equivariant $\Spinc$-structure with spinor bundle $S_N \to N$ such that
\[
S|_N = S_N \otimes S_{\kp}.
\]
So
\beq{eq decomp SW}
\Gamma^{\infty}(S \otimes W) \cong \bigl(C^{\infty}(G) \otimes S_{\kp} \otimes \Gamma^{\infty}(S_N \otimes W|_N) \bigr)^K.
\eeq
The $K$-invariant inner product on $\kg$ chosen earlier, together with the restriction of the Riemannian metric on $M$ to $TN$, defines a $K$-invariant metric on $TM|_N = TN \oplus (N \times \kp)$. We denote the extension of this metric to a $G$-invariant Riemannian metric on $M$ by $B_{\kp}$. By Lemma 3.12 in \cite{HSII}, the fact that $M$ is complete in the original Riemannian metric implies that it is complete with respect to $B_{\kp}$. Furthermore, the equivariant indices of the Dirac operators corresponding to the two metrics are the same. Indeed, the $K$-homology classes defined by these operators are the same, see Proposition 11.2.7 in \cite{Higson00}. 
Therefore, there is no loss of generality in working with $B_{\kp}$. 

Let $\{X_1, \ldots, X_l\}$ be an orthonormal basis of $\kp$ with respect to the chosen inner product. Let $D_{G,K}$ be the operator on $C^{\infty}(G) \otimes S_{\kp}$ defined as
\[
D_{G,K} := \sum_{j=1}^l L(X_{j}) \otimes c_{\kp}(X_j),
\]
where in the second factor, $c_{\kp} \colon \kp \to \End(S_{\kp})$ is the Clifford action with respect to the given inner product. Let $\varepsilon$ be the grading operator on $S_{\kp}$. 
By Proposition 3.1 in \cite{HSIII}, there is a $\Spinc$-Dirac operator $D_N$ on $\Gamma^{\infty}(S_N)$ such that $D$ is the operator
\beq{eq dec Dirac}
D_{G,K} \otimes 1 + \varepsilon \otimes D_N
\eeq
on 
 \eqref{eq decomp SW}.
Here we use the metric $B_{\kp}$.

Let 
\[
\begin{split}
\kappa_t &\in \Gamma^{\infty}(M\times M, \Hom(S\otimes W));\\
\kappa^{G,K}_t & \in C^{\infty}(G\times G) \otimes \End(S_{\kp});\\
\kappa^{N}_t & \in \Gamma^{\infty}(N\times N, \Hom(S_N \otimes W|_N)).
\end{split}
\]
be the Schwartz kernels of the operators $e^{-tD^2}$, $e^{-t D_{G,K}^2}$ and $e^{-t D_N^2}$, respectively. 

\begin{lemma}\label{lem decomp heat}  For all $x, x' \in G$ and $n, n' \in N$, we have
\begin{multline*}
\kappa_t(xn, x'n') = \kappa^{G,K}_t(x, x')\otimes x\kappa^N_t(n, n')x'^{-1} \\
 \in 
 \Hom((S\otimes W)_{x'n'}, (S\otimes W)_{xn}).
\end{multline*}
\end{lemma}
\begin{proof}
The presence of the grading operator in \eqref{eq dec Dirac} means that
\[
D^2 = D_{G,K}^2 \otimes 1 + 1 \otimes D_N^2.
\]
Since the two terms on the right hand side commute, we therefore have
\beq{eq decomp heat}
e^{-tD^2} = e^{-t(D_{G,K}^2 \otimes 1 + 1 \otimes D_N^2)} = e^{-tD_{G,K}^2} \otimes e^{-tD_N^2}.
\eeq

Lemma 4.1 in \cite{HSIII} states that the Riemanian density $dm$ on $M$ equals the measure $d[x, n]$ on $G\times_K N$ induced by the product measure $dx \, dn$ on $G\times N$. Since $K$ has unit volume, this implies that for all
$\varphi \in C^{\infty}(G) \otimes S_{\kp} \cap L^2(G) \otimes S_{\kp}$ and $\psi \in \Gamma^{\infty}(S_N \otimes W|_N)$ such that
$\varphi \otimes \psi$ is $K$-invariant, 
and $x \in G$ and $n \in N$,
\[
\begin{split}(e^{-tD^2} \varphi \otimes \psi)(xn) &= 
\int_M \kappa_t(xn, m') (\varphi\otimes \psi)(m') dm' \\
&=
  \int_N \int_{G} \kappa_t(xn, x'n') \varphi(x')\otimes x' \psi(n') \, dx'\, dn' 
\end{split}
\]
By \eqref{eq decomp heat}, this equals
\[
\begin{split}(e^{-tD_{G,K}^2}\varphi \otimes e^{-tD_N^2}\psi)(xn) &= 
\int_G \kappa^{G,K}_t(x, x')\varphi(x') \, dx' \otimes \int_Nx \kappa^N_t(n, n') \psi(n')\, dn'.
\end{split}
\]
Since $\varphi \in L^2(G) \otimes S_{\kp}$ and $N$ is compact, all integrals converge absolutely. So the claim follows from Fubini's theorem.
\end{proof}

\begin{proofof}{Lemma \ref{lem heat SE}}
Let $\tilde \kappa^{G,K}_t \in C^{\infty}(G) \otimes \End(S_{\kp})$ be given by
\[
\tilde \kappa^{G,K}_t(x) = x\kappa^{G,K}_t(e,x)x^{-1}.
\]
By Proposition 2.4 in \cite{BM}, this function lies in $\cC(G) \otimes \End(S_{\kp})$. So the claim follows by 
 Lemma \ref{lem decomp heat}.
\end{proofof}

\subsection{Finite Gaussian orbital integrals}

Consider the Gaussian function $\psi$ on $G$ defined in \eqref{eq def psi}.
 \begin{definition}\label{def FGOI}
 An element $g\in G$ has \emph{finite Gaussian orbital integral}, or \emph{FGOI}, if the integral
 \[
 \int_{G/Z} \psi(xgx^{-1})\, d(xZ)
 \]
 converges. Here, as before, $Z = Z_G(g)$.
 \end{definition}
 One immediately sees that $e\in G$ has FGOI. More generally, any element $g\in G$ such that $G/Z_G(g)$ is compact has FGOI. But actually, having FGOI is a generic property.
\begin{proposition} \label{prop FGOI ae}
Almost every element of $G$ has FGOI.
\end{proposition}

To prove this proposition, we first note the following.
\begin{lemma}\label{lem varphi L1}
The function $\psi$ is in $L^1(G)$.
\end{lemma}
\begin{proof}
Using the decomposition $G = KA^+K$, we have
\[
\int_G \psi(x) \, dx = \int_K \int_{A^+} \int_K \psi(kak') J(a) \, dk\, da\, dk',
\]
where for all $a \in A$,
\[
J(a) := \prod_{\alpha \in \Sigma^+} | e^{\langle \alpha, \log a \rangle } - e^{-\langle\alpha, \log a\rangle} |^{m_{\alpha}},
\]
with $\Sigma^+$ a choice of positive restricted roots for $(\kg, \ka)$, and $m_{\alpha} := \dim \kg_{\alpha}$. Now for all $a \in A$, $d(e,a) = \|\log(a)\|$.  So
\[
\psi(a) = e^{-\|\log(a)\|^2}.
\]
Therefore, $\psi J \in L^1(A)$, because $J$ only has linear functions of $\log(a)$ in its exponents. And for all $k, k' \in K$, we have
\[
d(e, kak') \geq d(e,a) -2\diam(K).
\]
(We write $\diam(X)$ for the diameter of a bounded metric space $X$.)
So $\psi(kak') \leq e^{-2\diam(K)}\psi(a)$. Hence $a \mapsto \psi(kak')J(a)$ is in $L^1(A)$, and the claim follows.
\end{proof}

\noindent \emph{Proof of Proposition \ref{prop FGOI ae}.}
By Lemma \ref{lem varphi L1} and Weyl's integration formula (see e.g. Proposition 5.27 in \cite{Knapp}), we have
\beq{eq WIF}
\infty > \int_G \psi(x)\, dx = \sum_H \frac{1}{\# W(G,H)} \int_{G/H\times H} \psi(xhx^{-1}) |D_H(h)|^2 dh\, d(xH).
\eeq
Here the sum ranges over  representatives $H<G$ that are invariant under the Cartan involution of the conjugacy classes of Cartan subgroups. For each such $H$, $W(G,H)$ is the corresponding Weyl group, and $D_H$ is the Weyl denominator. For a fixed $H$, let $\tilde H \subset H$ be the set of elements $h\in H$ for which the integral
\beq{eq int GH}
 \int_{G/H} \psi(xhx^{-1}) d(xH)
\eeq
converges. Then \eqref{eq WIF} implies that $H\setminus \tilde H$ has measure zero. Now let $h \in G$, and suppose it is regular. Then $H := Z_G(h)$ is a Cartan subgroup. (We can ensure that $H$ is invariant under the Cartan involution by conjugating by a group element; this does not change the integral \eqref{eq int GH}.)  We conclude that $h$ has FGOI if 
\[
h\in H^{\reg}\cap \tilde H,
\]
so the set of elements that do not have FGOI has measure zero.
\hfill $\square$

\subsection{An estimate on $G$}

As in the proof of Lemma \ref{lem heat SE}, let $\tilde \kappa^{G,K}_t \in C^{\infty}(G) \otimes \End(S_{\kp})$ be given by
\[
\tilde \kappa^{G,K}_t(x) = x\kappa^{G,K}_t(e,x)x^{-1},
\]
for $x \in G$.
There are constants $C,s>0$ such that for all $t \in {]0,s]}$,
\beq{eq est ht exp}
\|\tilde \kappa^{G,K}_t(x)\| \leq Ct^{-\dim(G)/2}e^{-\frac{d(e,x)^2}{16t}}.
\eeq
For heat kernels associated to the Laplace--Beltrami operator on scalar functions, this is
Theorem 4 in \cite{CLY81}. In (2.6) in \cite{BM}, Barbasch and Moscovici express the heat kernel of a Dirac operator on $G$ on terms of the heat kernel of the Laplacian on scalar functions. This gives the desired estimate.

Consider the relatively compact subset
\beq{eq def V}
V := \bigl\{x \in G; d(x, K) < 2\sqrt{2} \dim(G)^{1/2}\bigr\} \subset G.
\eeq
Define the map $\tilde \chi\colon G \to G$ by
\[
\tilde \chi(x) = xgx^{-1},
\]
for $x \in G$, and let $\chi\colon G/Z \to G$ by the induced map on the quotient. For an odd endomorphism $A$ of $S_{\kp}$, we will denote its supertrace by $\Str(A)$. 
\begin{lemma}\label{lemma est ht}
Suppose $g$ has FGOI. Then
we have
\[
\lim_{t\downarrow 0} \int_{(G/Z)\setminus \chi^{-1}(V)} |\Str \tilde \kappa^{G,K}_t(\tilde \chi(x))|d(xZ) = 0.
\]
\end{lemma}
\begin{proof}
For $t>0$, consider the function $\psi_t$ on $G$, given by
\[
\psi_t(x) = C t^{-\dim(G)/2}e^{-\frac{d(e,x)^2}{16t}},
\]
the right hand side of \eqref{eq est ht exp}. By Lemma \ref{lem calculus} below, the function $t\mapsto \psi_t(x)$ is increasing on $]0,1]$ for all $x \in G$ satisfying 
\[
e^{-\frac{d(e,x)^2}{16} } < e^{-\dim(G)/2},
\]
i.e.\ 
\[
d(e,x) > 2\sqrt{2}\dim(G)^{1/2}.
\]
Therefore, we have for all $t \in {]0,1]}$,
\[
\psi_t|_{G\setminus V} \leq \psi_1|_{G\setminus V},
\]
so for all $t \in {]0,\min(s,1)]}$,
\[
|\Str \tilde \kappa^{G,K}_t|_{G \setminus V}| \leq \psi_1|_{G\setminus V}.
\]
This implies that for all $xZ \in (G/Z) \setminus  \chi^{-1}(V)$,
\[
|\Str \tilde \kappa^{G,K}_t(\tilde \chi(x))| \leq \psi_1(\tilde \chi(x)).
\]
Since $g$ has FGOI, the function $x \mapsto \psi_1(\tilde \chi(x))$ is in $L^1(G/Z)$. Since $\tilde \kappa^{G,K}_t \to 0$ pointwise away from $e$, 
 the claim follows by the dominated convergence theorem.
\end{proof}

\begin{lemma} \label{lem calculus}
Let $b \in {]0,1]}$ and $c>0$. Consider the function $f\colon {]0,\infty[}\to \R$ given by
\[
f(t) = t^{-c}b^{1/t},
\]
for $t>0$. Then $f$ is increasing on $]0,1]$ if $b < e^{-c}$.
\end{lemma}
\begin{proof}
Under the condition stated, one has $f'(t)>0$ for all $t \in {]0,1]}$.
\end{proof}

\subsection{An estimate on $M$} \label{sec est M}

We now come to the most important step in the proof of Theorem \ref{thm fixed pt}.
\begin{proposition}\label{prop est kM}
Suppose $g$ has FGOI.
Let $c^g$ be as in part (b) of Lemma \ref{lem Trg conv}. There is a relatively $Z$-cocompact, $Z$-invariant neighbourhood $U$ of $M^g$ in $M$ such that for all $Z$-invariant neighbourhoods $U'$ of $M^g$ contained in $U$, we have
\[
\lim_{t \downarrow 0} \int_{M \setminus U'} c^g(m)|\Str \kappa_t(m, gm)g|\, dm = 0.
\]
\end{proposition}

Let $V \subset G$ be as in \eqref{eq def V}. We claim that the set
\beq{eq def U}
U := \tilde \chi^{-1}(V)^{-1}N \subset M
\eeq
has the properties needed in Proposition \ref{prop est kM}.
\begin{lemma}\label{lem nghbhd U}
The set $U$ is a relatively $Z$-cocompact neighbourhood of $M^g$.
\end{lemma}
\begin{proof}
The conjugacy class $(g) = \tilde \chi(G)$ is closed in $G$. Hence $V \cap (g)$ is a relatively compact subset of $(g)$. Since $\chi\colon G/Z \to (g)$ is a diffeomorphism, the inverse image $\chi^{-1}(V)  = \chi^{-1}(V \cap (g))\subset G/Z$ is relatively compact. So $\tilde \chi^{-1}(V)$ is relatively $Z$-cocompact with respect to right multiplication. Hence $\tilde \chi^{-1}(V)^{-1}$ relatively $Z$-cocompact with respect to left multiplication. By 
 compactness of $N$, this implies that  $U$ is relatively cocompact for the action by $Z$ on $M$.

Furthermore, let $x\in G$ and $n \in N$, and suppose that $gxn = xn$. Then 
\[
\tilde \chi(x^{-1}) =
x^{-1}gx \in G_n = K_n \subset K \subset V.
\]
So $xn \in U$.
\end{proof}

\begin{lemma}\label{lem int cg}
Let $c^g$ be as in part (b) of Lemma \ref{lem Trg conv}. Let $X \subset G$ be a subset invariant under right multiplication by $Z$. Then for all $\tilde \kappa \in \tilde \cC(E)$,
\beq{eq int cg Z}
\int_{M\setminus X^{-1}N} c^g(m)|\Str(\kappa(m, gm)g)|\, dm = \int_N \int_{(G\setminus X)/Z} |\Str(\tilde \kappa(xgx^{-1}, n, n))| \, d(xZ)\, dn.
\eeq
\end{lemma}
\begin{proof}
The proof is a computation that is as straightforward as the computation in the proof of Lemma \ref{lem Trg conv}, but different in some respects. We give the computation here to be complete. The left hand side of \eqref{eq int cg Z} equals
\beq{eq int A}
\int_{N} \int_{G \setminus X^{-1}} \int_G c_G(x)c(xgyn)|\Str(\kappa(yn, gyn)g)|\, dx\, dy\, dn. 
\eeq
As in \eqref{eq traces}, we have for all $y \in G$ and $n \in N$,
\[
\Str(\kappa(yn, gyn)g) = \Str(\tilde \kappa(y^{-1}gy, n,n)).
\]
So by unimodularity of $G$, \eqref{eq int A} equals
\[
\int_{N} \int_{G \setminus X} \int_G c_G(x^{-1})c(x^{-1}gy^{-1}n)| \Str(\tilde \kappa(ygy^{-1}, n,n))|\, dx\, dy\, dn. 
\]
Interchanging the inner two integrals (which is allowed because the integrals converge absolutely as shown in the proof of Lemma \ref{lem Trg conv}), and
substituting $x' = x^{-1}gy^{-1}$ for $x$ and $y' = x'yg^{-1}$ for $y$, we find that the above integral equals
\[
\int_{N}\int_G \int_{x'(G \setminus X^{-1})g^{-1}}  c_G(y')c(x'n)| \Str(\tilde \kappa(x'^{-1}y'gy'^{-1}x', n,n))|\, dy'\, dx'\,  dn. 
\]
The sets $x'(G \setminus X)g^{-1}$ are invariant under right multiplication by $Z$, so since $c_G$ is a cutoff function for that action, the integral becomes
\begin{multline*}
\int_{N}\int_G \int_{(x'(G \setminus X)g^{-1})/Z}  c(x'n)| \Str(\tilde \kappa(x'^{-1}y'gy'^{-1}x', n,n))|\, d(y'Z)\, dx'\,  dn \\
= \int_{N}\int_G \int_{(G \setminus X)/Z}  c(x'n)| \Str(\tilde \kappa(ygy^{-1}, n,n))|\, d(yZ)\, dx'\,  dn \\
=\int_{N} \int_{(G \setminus X)/Z}  | \Str(\tilde \kappa(ygy^{-1}, n,n))|\, d(yZ)\,  dn.
\end{multline*}
\end{proof}


\begin{lemma}\label{lem est kM}
Suppose $g$ has FGOI.
Let $c^g$ be as in part (b) of Lemma \ref{lem Trg conv}. Then
\[
\lim_{t \downarrow 0} \int_{M \setminus U} c^g(m)|\Str \kappa_t(m, gm)g|\, dm = 0.
\]
\end{lemma}
\begin{proof}
By Lemma \ref{lem int cg}, we have
\[
\int_{M \setminus U} c^g(m)|\Str \kappa_t(m, gm)g|\, dm 
= \int_N \int_{(G/Z)\setminus \chi^{-1}(V)} |\Str \tilde \kappa_t(xgx^{-1}, n, n)|\, d(xZ)\, dn. 
\]
By Lemma \ref{lem decomp heat}, this equals 
\[
\int_N  |\Str \kappa^N_t(n, n)| \, dn \int_{(G/Z)\setminus  \chi^{-1}(V)} |\Str  \tilde \kappa^{G,K}_t(\tilde \chi(x)) | \, d(xZ).
\]
Note that $S_N$ is $\Z_2$-graded because $M$ and $G/K$, and therefore $N$, are even-dimensional. This is why a supertrace appears in the first factor.
This first factor is bounded, so the claim follows by Lemma \ref{lemma est ht}.
\end{proof}

\begin{lemma}\label{lem cg cutoff}
The function $c^g$ as in part (b) of Lemma \ref{lem Trg conv} is a cutoff function for the action by $Z$ on $M$.
\end{lemma}
\begin{proof}
For all $m \in M$, we have
\[
\begin{split}
\int_Z c^g(zm) = \int_Z \int_G c_G(x) c(xgzm) \, dx\, dz = \int_Z \int_G c_G(x'z'^{-1}) c(x'm) \, dx'\, dz' = 1,
\end{split}
\]
where we substitute $x' = xgz$ for $x$ and $z' = gz$ for $z$, and use unimodularity of $G$ and $Z$.
\end{proof}

\noindent \emph{Proof of Proposition \ref{prop est kM}.}
The set $U$ is a $Z$-cocompact neighbourhood of $M^g$ by Lemma \ref{lem nghbhd U}. Let $U'$ be a $Z$-invariant neighbourhood of $M^g$ contained in $U$. By Lemma \ref{lem est kM}, it is enough to show that
\beq{eq est U U'}
\lim_{t \downarrow 0} \int_{\overline{U} \setminus U'} c^g(m)|\Str \kappa_t(m, gm)g|\, dm = 0.
\eeq

It is a basic property of heat kernels that $\kappa_t(m, m')$ goes to zero as $t\downarrow 0$, for all distinct points $m, m' \in M$. So $|\Str \kappa_t(m, gm)g|$ goes to zero pointwise as $t \downarrow 0$ for all $m \in M \setminus M^g$.
The function
\beq{eq bd fn}
m \mapsto \| \kappa_t(m, gm) \|
\eeq
is $Z$-invariant, and $|\Str \kappa_t(m, gm)g|$ is at most equal to this function times the rank of $E$ at every point $m \in M$.
Since $\overline{U} \setminus U'$ is disjoint from $M^g$ and $Z$-cocompact, the bounding function \eqref{eq bd fn} is uniformly bounded in $t$ on this set, say by a constant $C>0$. Since $\overline{U} \setminus U'$ is  $Z$-cocompact and, by Lemma \ref{lem cg cutoff}, $c^g$ is a cutoff function for the action by $Z$, the integral
\[
\int_{\overline{U} \setminus U'} c^g(m)\, dm
\]
converges. Hence \eqref{eq est U U'} follows by the dominated convergence theorem. 
\hfill $\square$

\subsection{Proof of Theorem \ref{thm fixed pt}} \label{sec proof}

Proposition \ref{prop est kM} allows us to prove Theorem \ref{thm fixed pt} for twisted $\Spinc$-Dirac operators. This implies the general case.

Let $L_{\det} \to M$ be the determinant line bundle of the $\Spinc$-structure on $M$. Let $R^{\cN}$ be the curvature of the Levi--Civita connection restricted to $\cN$. Let $\hat A(M^g)$ be the $\hat A$-class of $M^g$. Let $[W|_{\supp(c^g)}](g) \in K^0(M^g) \otimes \C$ be defined as in Subsection \ref{sec result}.
\begin{proposition} \label{prop fixed Spinc}
Suppose that $D$ is a twisted $\Spinc$-Dirac operator on $E = S \otimes W$. 
If $G/K$ is odd-dimensional, then for all semisimple $g \in G$,
\[
\tau_g(\ind_G(D)) = 0.
\] 
If $G/K$ is even-dimensional, then
for  all semisimple $g \in G$ with FGOI, we have
\[
\tau_g(\ind_G(D)) = 0
\]
if $g$ is not contained in a compact subgroup of $G$, and 
\beq{eq fixed Spinc}
\tau_g(\ind_G(D)) = \int_{M^g} c^g \frac{\hat A(M^g) \ch([W|_{M^g}](g)) e^{c_1(L_{\det}|_{M^g})}}{\det(1-g e^{-R^{\cN}/2\pi i})^{1/2}},
\eeq
if it is, for any cutoff function $c^g$ for the action by $Z$ on $M^g$.
%
\end{proposition}
%
\begin{proof}
Let $g \in G$ be semisimple.
First of all, if $G/K$ is odd-dimensional, then $K_0(C^*_rG) = 0$. This is part of the statement of the Connes--Kasparov conjecture, proved in \cite{CEN, Wassermann87}. Since $\tau_g$ is defined on $K_0(C^*_rG)$, this implies that $\tau_g(\ind_G(D)) = 0$. So suppose from now on that $G/K$ is even-dimensional.

By Proposition \ref{prop taug index}, we have for all $t>0$
\beq{eq comp index}
\tau_g(\ind_G(D)) =  \int_M  c^g(m) \Str(\kappa_t(m, gm) g)\, dm,
\eeq
with $c^g$ as in part (b) of Lemma \ref{lem Trg conv}.
If $M^g$ is nonempty, then there is a point $m \in M$ fixed by $g$. So $g \in G_m$, which is a compact group by properness of the action. 
So if $g$  is  not contained in a compact subgroup of $G$ and has FGOI, then $M^g = \emptyset$, and one may take $U' = \emptyset$ in Proposition \ref{prop est kM}. This proposition then implies that the right hand side of \eqref{eq comp index} goes to zero as $t \downarrow 0$.

Now suppose that $g$ has FGOI and is contained in a compact subgroup of $G$. Then Proposition \ref{prop est kM} implies that the right hand side of \eqref{eq comp index} localises to arbitrarily small neighbourhoods of $M^g$ as $t \downarrow 0$. We choose a small enough tubular neighbourhood, which we identify with $\cN$, and apply the usual asymptotic expansion of $\kappa_t$ (see e.g.\ Theorem 6.11 in \cite{BGV}). Then by the same arguments as in the compact case (see e.g. Theorem 6.16 in \cite{BGV}), one obtains the desired expression for the right hand side of \eqref{eq comp index}. 

Here we used that $c^g|_{M^g}$ is a cutoff function for the action by $Z$ on $M^g$ by Lemma \ref{lem cg cutoff}. We have proved \eqref{eq fixed Spinc} for cutoff functions of this form. But since the integrand on the right hand side of \eqref{eq fixed Spinc} without $c^g$ is $Z$-invariant, the integral is independent of the cutoff function (because it equals an integral on $M^g/Z$ that does not involve a cutoff function). So the claim follows for all cutoff functions.
\end{proof}

\begin{proofof}{Theorem \ref{thm fixed pt}}
We return to the general case where $D$ is an elliptic operator, as at the start of Section \ref{sec prelim}.
Let $p_B \colon BM \to M$ be the unit ball bundle in $TM$, and $SM \to M$ the unit sphere bundle. Consider the almost complex manifold $\Sigma M$ obtained by glueing together two copies of $BM$ along $SM$. 
Let $E_{\pm}$ be the even and odd parts of $E$, respectively. Let $W_{\sigma_D} \to \Sigma M$ be the vector bundle obtained by glueing together $p_B^{*}E_+ \to BM$ and $p_B^* E_- \to BM$ along $SM$ via the invertible map $\sigma_D\colon SM \to SM$.
Let $D_{\Sigma M}^{W_{\sigma_D}}$ be the $\Spinc$-Dirac operator on $\Sigma M$ (for the $\Spinc$-structure associated to the almost complex structure), twisted by $W_{\sigma_D}$. Let $[D_{\Sigma M}^{W_{\sigma_D}}] \in K_0^G(\Sigma M)$ be its $K$-homology class.

Let $p_{\Sigma}\colon \Sigma M \to M$ be the projection map, and note that it is proper.
By Theorem 5.0.4 in \cite{BvEII}, we have
\[
[D] = (p_{\Sigma})_*[D_{\Sigma M}^{W_{\sigma_D}}] \quad \in K_0^G(M).
\]
Naturality of the assembly map therefore implies that
\[
\ind_{G}(D) = \ind_{G}(D_{\Sigma M}^{W_{\sigma_D}}). 
\]
Therefore, Proposition \ref{prop fixed Spinc} implies that $\tau_g (\ind_G(D)) = 0$ for all semisimple $g \in G$ if $G/K$ is odd-dimensional. So suppose that $G/K$ is even-dimensional. Let $g \in G$ be semisimple with FGOI. By Proposition \ref{prop FGOI ae}, almost every element of $G$ has FGOI. If 
 $g$ does not lie in a compact subgroup of $G$, then Proposition \ref{prop fixed Spinc} again implies that $\tau_g (\ind_G(D)) = 0$. 
 And if  $g$ lies in a compact subgroup of $G$, then this proposition states that
\[
\tau_g (\ind_{G}(D)) =  \int_{\Sigma M^g} c_{\Sigma M}^g \frac{\hat A((\Sigma M)^g) \ch([W_{\sigma_D}|_{(\Sigma M)^g}](g) )e^{c_1(L_{\det}|_{(\Sigma M)^g})}}{\det(1-g e^{-R^{\cN_{(\Sigma M)^g}}/2\pi i})^{1/2}}
\]
Here $\cN_{(\Sigma M)^g} \to (\Sigma M)^g$ is the normal bundle, and $c_{\Sigma M}^g \in C^{\infty}_c( (\Sigma M)^g)$ is a cutoff function for the action by $Z$ on $(\Sigma M)^g$.  As in the compact case, one finds that the right hand side of the above expression equals the right hand side of \eqref{eq fixed pt}. (We may choose the cutoff function $c_{\Sigma M}^g$ to be constant on the fibres of $(\Sigma M)^g\to M^g$, so it reduces to a cutoff function on $M^g$ as in Theorem \ref{thm fixed pt}.) So Theorem \ref{thm fixed pt} follows.
%
\end{proofof}


\section{The character formula} \label{sec char}

We end this paper by deducing Harish-Chandra's character formula, Corollary \ref{cor char}, from Theorem \ref{thm fixed pt}. We first discuss $K$-theory classes defined by discrete series representations. Then we show that characters of these representations can be recovered from their $K$-theory classes via the trace $\tau_g$. By realising these $K$-theory classes as equivariant indices, we can then apply Theorem \ref{thm fixed pt} to obtain Corollary \ref{cor char}.

Throughout this section, we make the assumptions and use the same notation as in  Subsection \ref{sec char form}. In particular, we suppose that $\rank(G) = \rank(K)$, and we consider a maximal torus $T<K$, and a discrete series representation $\pi$ of $G$ with Harish-Chandra parameter $\lambda \in i\kt^*$.

\subsection{$K$-theory classes and characters}

Let $\xi$ be a $K$-finite unit vector in the representation space of $\pi$. Let $m_{\xi}$ be the corresponding matrix coefficient, mapping $g \in G$ to 
\[
m_{\xi}(g) = (\xi, \pi(g)\xi).
\]
Then $m_{\xi} \in \cC(G)$. Let 
\[
d_{\pi} = \|m_{\xi}\|_{L^2(G)}^{-2}
\]
be the formal degree of $\pi$. Then $d_{\pi}m_{\xi}$ is an idempotent in $\cC(G)$ with respect to convolution. So it defines a class
\[
[d_{\pi}m_{\xi}] \in K_0(\cC(G)).
\]
(See also Subsection 2.2 in \cite{Lafforgue02}.)

The values on $T$ of the global character $\Theta_{\pi}$ of $\pi$ can be recovered from this $K$-theory class. 
\begin{proposition}\label{prop taug char}
Let $g \in T$ be regular. Then
\[
\tau_g([d_{\pi}m_{\xi}]) = 
\Theta_{\pi}(g). 
\]
\end{proposition}
\begin{proof}
Let $\varphi \in C^{\infty}_c(G)$ be supported in the set of regular elliptic elements (i.e.\ the set of elements whose centraliser is a conjugate of $T$). Let $\pi(\varphi)$ be the trace-class operator on the representation space of $\pi$ defined by $\varphi$. Then the claim is that
\beq{eq taug char}
\tr(\pi(\varphi)) = \int_G \varphi(g) \tau_g([d_{\pi}m_{\xi}])\, dg.
\eeq
By Proposition 14.4.3 in \cite{Dixmier}, which is proved there via basic functional analysis, we have
\beq{eq dix}
\tr(\pi(\varphi)) = d_{\pi} \int_G \int_G \varphi(g) m_{\xi}(xgx^{-1}) \, dg\, dx.
\eeq
For general $\varphi \in C^{\infty}_c(G)$, the integral on the right hand side may diverge if the order of integration is reversed. But
if $g \in G$ is elliptic and regular, then $Z = Z_G(g)$ is conjugate to $T$, hence compact. Then by  Theorem \ref{thm orbital integral}, the integral 
\[
\int_G   m_{\xi}(xgx^{-1}) \, dx = \int_{G/Z}   m_{\xi}(xgx^{-1}) \, d(xZ).
\]
converges absolutely.
Since $\varphi$ is compactly supported inside the set of elliptic elements,
the integral
\[
 \int_G \int_G \varphi(g) m_{\xi}(xgx^{-1}) \, dx\, dg =
 \int_G \int_{G/Z} \varphi(g) m_{\xi}(xgx^{-1}) \, d(xZ) \, dg 
\] 
 converges absolutely as well.
So by Fubini's theorem, it equals
\[
 \int_G \int_G \varphi(g) m_{\xi}(xgx^{-1}) \, dg\, dx.
\]
Since $m_{\xi} \in \cC(G)$, we have
\[
 \tau_g([d_{\pi}m_{\xi}]) = d_{\pi} \int_{G/Z} m_{\xi}(xgx^{-1}) \, d(xZ).
\]
Hence  \eqref{eq dix} implies \eqref{eq taug char}.
\end{proof}

\subsection{$K$-theory classes  as indices}\label{sec ds ind}

The $K$-theory class $[d_{\pi}m_{\xi}] \in K_0(\cC(G))$ can be realised geometrically analogously to Schmid's realisation of the discrete series in Theorem 1.5 in \cite{Schmid76}. Consider the manifold $G/T$. The positive root system $R^+$ determined by $\lambda$ defines a $G$-invariant complex structure on $G/T$ such that, as complex vector spaces,
\[
T_{eT}(G/T) = \bigoplus_{\alpha \in R^+}\kg^{\C}_{\alpha}.
\]
Consider the holomorphic line bundle
\[
L_{\lambda - \rho} = G\times_T \C_{\lambda - \rho} \to G/T,
\]
where $T$ acts on $\C_{\lambda - \rho} = \C$ with weight $\lambda - \rho$. We have the twisted Dolbeault operator
$
\bar\partial_{L_{\lambda-\rho}}
$ on $\Omega^{0,*}(G/T; L_{\lambda-\rho})$. On the space $\Omega^{0,*}_{L^2}(G/T; L_{\lambda-\rho})$ of square integrable antiholomorphic differential forms with coefficients in $L_{\lambda-\rho}$, we have the adjoint operator $\bar\partial_{L_{\lambda-\rho}}^*$. These two operators combine to the Dolbeault--Dirac operator
\[
\bar\partial_{L_{\lambda-\rho}} + \bar\partial_{L_{\lambda-\rho}}^*
\]
on $\Omega^{0,*}_{L^2}(G/T; L_{\lambda-\rho})$. We consider the natural grading on forms by parities of degrees.
\begin{proposition} \label{prop ds index}
We have
\[
[d_{\pi}m_{\xi}] = (-1)^{\dim(G/K)/2}\ind_G(\bar\partial_{L_{\lambda-\rho}} + \bar\partial_{L_{\lambda-\rho}}^*) \quad \in K_0(\cC(G)).
\]
\end{proposition}
This result is  Corollary 2.8 in \cite{Hochs15}. However, in that paper, Lemma 1.5 in \cite{Hochs09} is used to prove Corollary 2.8, and  in the proof of Lemma 1.5 in \cite{Hochs09}, Harish-Chandra's character formula is used. So this would lead to a circular argument. To get around this, we give a proof of Proposition \ref{prop ds index} in Subsections \ref{sec Ind} and \ref{sec pf ds index} without using Lemma 1.5 in \cite{Hochs09}.  This also gives us the opportunity to record a short proof of an earlier induction result for the equivariant index, Theorem \ref{thm Q Ind}. Furthermore, we clarify the case where $G/K$ does not have a $G$-equivariant $\Spin$ structure, which was not treated in \cite{Hochs09}.

\begin{remark}
If $G$ is linear, then Proposition \ref{prop ds index} also follows from Schmid's construction of the discrete series.
Theorem 1.5 in \cite{Schmid76} implies that the kernel of  $\bar\partial_{L_{\lambda-\rho}} + \bar\partial_{L_{\lambda-\rho}}^*$ equals the representation space of $\pi$ in degree $(-1)^{\dim(G/K)/2}$, and zero in the other degree. Then the equivariant index of $\bar\partial_{L_{\lambda-\rho}} + \bar\partial_{L_{\lambda-\rho}}^*$ in $K_0(\cC(G)) = K_0(C^*_rG)$ equals $(-1)^{\dim(G/K)/2}$ times the kernel of this operator as a $C^*_rG$-module, see Lemma II.10.$\gamma$.16 in \cite{ConnesBook}. This kernel is $\pi$, which is also the image of the projection $d_{\pi}m_{\xi}$.
\end{remark}

After Proposition \ref{prop ds index} is proved, we can deduce Corollary \ref{cor char} from Theorem \ref{thm fixed pt} as follows.

\medskip
\begin{proofof}{Corollary \ref{cor char}}
Let $g \in T$ be regular, and in the set of elements for which Theorem \ref{thm fixed pt} holds. Then Propositions \ref{prop taug char} and \ref{prop ds index}, and Theorem \ref{thm fixed pt}, imply that\footnote{Actually, we only need Proposition \ref{prop fixed Spinc} here, rather than Theorem \ref{thm fixed pt}, but then the application of the arguments in Subsection 6.5 of \cite{HW}  becomes  less direct.} 
\beq{eq fixed char}
\Theta_{\pi}(g) = (-1)^{\dim(G/K)} \int_{T(G/T)^g} \frac{\ch\bigl([\sigma_{\bar\partial_{L_{\lambda-\rho}} + \bar\partial_{L_{\lambda-\rho}}^*}|_{(G/T)^g}](g)\bigr)\Todd(T(G/T)^g \otimes \C)}{\ch\bigl([\Bigwedge \cN \otimes \C] (g) \bigr)}.
\eeq
Suppose that 
 the powers of $g$ are dense in $T$. Then
 the cutoff function $c^g$ may be taken to be constant 1, since $(G/T)^g = (G/T)^T$ is compact. It was shown in Subsection 6.5 of \cite{HW} that the right hand side of \eqref{eq fixed char} equals the right hand side of \eqref{eq char}. The set of elements $g \in T$ with the assumed properties is dense in the set of regular elements (or indeed, in the whole torus $T$). Since both sides of \eqref{eq char} are  analytic functions on the regular elements, the result follows.
\end{proofof}

\begin{remark}
In our proof of Harish-Chandra's character formula, we made essential use of $K$-theory. We could have stated and proved a fixed point theorem by defining the index as the right hand side of the equality in Proposition \ref{prop taug index Trg}. But without using $K$-theory, one does not immediately see that, for $g\in G$ and $t>0$,
\beq{eq theta tr}
  \Theta_{\pi}(g) = (-1)^{\dim(G/K)/2}\bigl( \Tr_g (e^{-t\bar \partial_{\lambda - \rho}^- \bar \partial_{\lambda - \rho}^+}) -  \Tr_g (e^{-t\bar \partial_{\lambda - \rho}^+ \bar \partial_{\lambda - \rho}^-}) \bigr).
\eeq
Indeed, if zero were isolated in the spectrum of $\bar \partial_{\lambda - \rho} + \bar \partial_{\lambda - \rho}^*$, then the heat kernel $e^{-t\bar \partial_{\lambda - \rho}^{\pm} \bar \partial_{\lambda - \rho}^{\mp}}$ converges to projection onto the $L^2$-kernel of $\bar \partial_{\lambda - \rho}^{\mp}$ as $t \to \infty$. By Theorem 1.5 in \cite{Schmid76}, if $\dim(G/K)/2$ is even, then the $L^2$-kernel of $\bar \partial_{\lambda - \rho}^{+}$
 is the representation space of $\pi$, while the $L^2$-kernel of $\bar \partial_{\lambda - \rho}^{-}$ is zero.  If $\dim(G/K)/2$ is odd, then this is the other way around. That would imply \eqref{eq theta tr}. However, it is not clear if zero is isolated in the spectrum of $\bar \partial_{\lambda - \rho} + \bar \partial_{\lambda - \rho}^*$, and using $K$-theory allows us to get around this issue.
\end{remark}

It remains to prove Proposition \ref{prop ds index}.

\subsection{Induction of indices} \label{sec Ind}

In this subsection only, let $G$ be any almost connected Lie group, and let $K<G$ be maximal compact.
Theorem 4.6 in \cite{Hochs09} is a result about the relation between equivariant indices for $K$ and $G$. (It was called \emph{quantisation commutes with induction} there.) This result was slightly expanded on and applied in \cite{Hochs15, HM16, HM14}. In Theorem 43 in 
\cite{GMW}, a shorter and more direct proof than the one in \cite{Hochs09} was given. Here we present an even shorter proof of this result. We also explain how it generalises to cases where $G/K$ has no $G$-equivariant $\Spin$-structure.

Let $\tilde K$ be a double cover of $K$ such that the adjoint action $\Ad\colon \tilde K \to \SO(\kp)$ lifts to a homomorphism $\widetilde{\Ad}\colon \tilde K \to \Spin(\kp)$. Let $R(\tilde K)_- \subset R(K)$ be the free abelian group generated by the irreducible representations of $\tilde K$ for which the nontrivial element of the kernel of $\tilde K \to K$ acts as minus the identity. Set $l := \dim(G/K)/2$. The Dirac induction map is an isomorphism of abelian groups
\[
\DInd_K^G\colon R(\tilde K)_- \to K_{l}(C^*_rG).
\]
See \cite{Wassermann87} and \cite{CEN} for the definition of this map and fact that it is an isomorphism. (We will not use the fact that this map is an isomorphism, however.)

Let $N$ be a compact manifold on which $K$ acts. Consider the equivariant $K$-homology groups \cite{Connes94, Higson00}  $K_0^{\tilde K}(N)$ and $K_0^G(G\times_K N)$  of the spaces $N$ and $G\times_K N$, respectively. Let 
$
K_0^{\tilde K}(N)_- \subset K_0^{\tilde K}(N)
$
be the abelian group
\[
K_0^{\tilde K}(N)_- = \{a \in K_0^{\tilde K}(N); \ind_{\tilde K}(a) \in R(\tilde K)_- \}.
\]
In Section 5 of \cite{Hochs09}, a map
\[
\KInd_K^G\colon K_0^K(N) \to K_l^G(G\times_K N)
\]
between equivariant $K$-homology groups
was defined in cases where the lift $\widetilde{\Ad}$ already exists for $K$ itself. This generalises directly to a map
\[
\KInd_K^G\colon K_0^{\tilde K}(N)_- \to K_l^G(G\times_K N).
\]
This map has the following properties that we will use.
\begin{enumerate}
\item If $N = \pt$ is a point, then for all $V \in R(\tilde K)_- = K_0^{\tilde K}(\pt)_-$, we have
\beq{eq KInd 1}
\KInd_K^G[V] = [D_{G,K}^V],
\eeq
the $K$-homology class of the Dirac operator on $G/K$ used in the definition of Dirac induction.
\item If $p$ is the map from $N$ to a point, and $p^G \colon G\times_K N \to G/K$ is the induced map, then the following diagram commutes:
\beq{eq KInd 2}
\xymatrix{
K_l^G(G\times_K N) \ar[r]^-{p^G_*}& K_l^G(G/K) \\
K_0^{\tilde K}(N)_- \ar[u]^-{\KInd_K^G} \ar[r]^{ p_*}& K_0^{\tilde K}(\pt)_-. \ar[u]_-{\KInd_K^G}
}
\eeq
\end{enumerate}

We can now state the induction result for equivariant indices.
\begin{theorem} \label{thm Q Ind}
The following diagram commutes:
\[
\xymatrix{
K_l^G(G\times_K N) \ar[r]^-{\ind_G}& K_l(C^*_rG) \\
K_0^{\tilde K}(N)_- \ar[u]^-{\KInd_K^G} \ar[r]^{\ind_{\tilde K}}& R(\tilde K)_-. \ar[u]_-{\DInd_K^G}
}
\]
\end{theorem}
\begin{proof}
By \eqref{eq KInd 1}, we have for all $V \in R(\tilde K)_-$,
\[
\ind_G(\KInd_K^G[V] ) = \DInd_K^G[V].
\]
This equality, together with 
commutativity of \eqref{eq KInd 2} and the fact that $p_*$ is the $\tilde K$-equivariant index, implies that for all $a \in K_0^{\tilde K}(N)_-$,
\[
\DInd_K^G(\ind_{\tilde K}(a)) = \ind_G \circ p^G_* \circ \KInd_K^G(a).
\]
By naturality of the assembly map, the following diagram commutes:
\[
\xymatrix{
K_l^G(G\times_K N) \ar[r]^-{\ind_G} \ar[d]_-{p^G_*}& K_l(C^*_rG). \\
K_l^G(G/K) \ar[ur]_-{\ind_G} &
}
\]
So the claim follows.
\end{proof}

\subsection{Proof of Proposition \ref{prop ds index}} \label{sec pf ds index}

We return to the setting of Subsection \ref{sec ds ind}, where $G$ is connected and semisimple with discrete series.
Let $\rho_c$ and $\rho_n$ be half the sums of the compact and noncompact roots in $R^+$, respectively. 
Let $V_{\lambda-\rho_c}$ be the irreducible representation of $K$ with highest weight $\lambda - \rho_c$. Since $(\lambda -\rho_c) - \rho_n$ is integral for $K$, the space $V_{\lambda - \rho_c}\otimes S_{\kp}$ has a well-defined representation of $K$, even if $S_{\kp}$ itself does not (see the bottom of page 21 in \cite{Atiyah77}). So it lifts to an element of
 $R(\tilde K)_-$.

 
\begin{lemma}\label{lem ds DInd}
We have
\[
[d_{\pi}m_{\xi}] = 
\DInd_K^G[V_{\lambda - \rho_c}] \quad \in K_0(C^*_rG).
\]
\end{lemma}
\begin{proof}
 Let $\calH_{\pi}$ be the representation space of $\pi$.
By Lemmas 2.1.1 and 2.2.1 in \cite{Lafforgue02}, and the fact that the Connes--Kasparov conjecture is true, we have
\[
[d_{\pi} m_{\xi}] =\bigl( \dim(\calH_{\pi} \otimes (S_{\kp}^+)^* \otimes V^*)^K -  \dim(\calH_{\pi} \otimes (S_{\kp}^-)^* \otimes V^*)^K \bigr)\DInd_K^G[V],
\]
for the unique $V \in \hat {\tilde K} \cap R(\tilde K)_-$ for which the right hand side is nonzero. (See also the comment below Lemma 2.2.1 in \cite{Lafforgue02}.) This was also noted on page 562 of \cite{Wassermann87}. Now if $V = V_{\lambda - \rho_c}$, then $S_{\kp}^+ \otimes V$ contains the irreducible representation $V_{\lambda - \rho_c + \rho_n}$ with highest weight $\lambda - \rho_c + \rho_n$ with multiplicity one, whereas $S_{\kp}^- \otimes V$ does not contain  $V_{\lambda - \rho_c + \rho_n}$. Since $V_{\lambda - \rho_c + \rho_n}$ is the lowest $K$-type of $\pi$, it occurs with multiplicity one in $\calH_{\pi}|_K$. All other irreducible constituents of $S_{\kp}^+ \otimes V_{\lambda - \rho_c}$ have highest weights lower than $\lambda - \rho_c + \rho_n$, in the ordering where positive roots are positive. So they do not occur in $\calH_{\pi}|_K$. Hence
\[
 \dim(\calH_{\pi} \otimes (S_{\kp}^+)^* \otimes V_{\lambda-\rho_c}^*)^K -  \dim(\calH_{\pi} \otimes (S_{\kp}^-)^* \otimes V_{\lambda - \rho_c}^*)^K = 1.
\]
%
\end{proof}

\begin{remark}
Interestingly, Lemma \ref{lem ds DInd} contradicts (5.3) in \cite{Hochs15} and Example (4.25) in \cite{Connes94}. 

The difference between Lemma \ref{lem ds DInd} and (5.3) in \cite{Hochs15} is a sign $(-1)^{\dim(G/K)/2}$. But (5.3) in \cite{Hochs15} is based on Lemma 1.5 in \cite{Hochs09}, where the same sign in the expression for the character of $S_{\kp}^*$ is missing. This is because there should not be a complex conjugate in the middle of the third display on page 873 of \cite{Hochs09}. See also (4.1) in \cite{Atiyah77}, where dualising $S_{\kp}$ leads to the introduction of this sign.

Example (4.25) in \cite{Connes94} is the  case  where $G = \SL(2,\R)$. For $n \in \Z_{\geq 1}$, let $\pi = D^+_n$ be the holomorphic discrete series representation with Harish-Chandra parameter $n\alpha/2$, where $\alpha$ is the standard positive root for the Cartan subgroup $T = K = \SO(2)$. For $k \in \Z$, let $V_{k}$ be the irreducible representation of $T$ with weight $k\alpha/2$, i.e.\ the standard representation of the circle on $\C$ with weight $k$. It is stated in Example (4.25) in \cite{Connes94} that
\[
[d_{\pi}m_{\xi}]  = - \DInd_K^G[V_{1-n}],
\]
whereas Lemma \ref{lem ds DInd} yields
\[
[d_{\pi}m_{\xi}]  = \DInd_K^G[V_{n}]
\]
in this case. 

The difference can be explained by two facts. The first is that in Example (4.25) in \cite{Connes94} it seems to be used that the $K$-types
 of $D^+_n$ are $V_n$, $V_{n+2}$, $V_{n+4}$ etc., whereas that should be  $V_{n+1}$, $V_{n+3}$, $V_{n+5}$ etc. This may be due to a different parametrisation of the discrete series, although the parameter $n$ runs over the positive integers as usual. The second fact is that in Example (4.25) in \cite{Connes94}, the space $(\calH_{\pi} \otimes S_{\kp}^{\pm}\otimes V_n)^K$ is used where we use $(\calH_{\pi} \otimes (S_{\kp}^{\pm})^*\otimes V_n^*)^K$. It seems to us that the space $(\calH_{\pi} \otimes (S_{\kp}^{\pm})^*\otimes V_n^*)^K$ is needed here, as in Lemma 2.1.1 in \cite{Lafforgue02}, in \cite{Wassermann87} and in (5.7) in \cite{Atiyah77}. This leads to Lemma \ref{lem ds DInd}, which is compatible with the fact that the kernel of the Dirac operator whose equivariant index is $\DInd_K^G[V_{\lambda - \rho_c}]$ is $\calH_{\pi}$ in even degree and zero in odd degree, as shown in \cite{Atiyah77, Parthasarathy72}. By Lemma II.10.$\gamma$.16 in \cite{ConnesBook}, this also implies Lemma \ref{lem ds DInd} in the form that we have stated it.
\end{remark}

Consider the $K$-invariant complex structure on $K/T$ defined by the compact roots in $R^+$. Consider the holomorphic line bundle
\[
L_{\lambda - \rho_c} = K\times_T \C_{\lambda - \rho_c} \to K/T.
\]
Let $\bar \partial^{K/T}_{L_{\lambda - \rho_c}}$ be the Dolbeault operator on $K/T$ coupled to $L_{\lambda - \rho_c}$.
\begin{lemma}\label{lem KInd Dolb}
We have
\beq{eq KInd Dolb}
\KInd_K^G[\bar \partial^{K/T}_{L_{\lambda - \rho_c}} + (\bar \partial^{K/T}_{L_{\lambda - \rho_c}})^*] = (-1)^{\dim(G/K)/2} [\bar\partial_{L_{\lambda-\rho}} + \bar\partial_{L_{\lambda-\rho}}^*] \quad \in K_0^G(G/T).
\eeq
\end{lemma}
\begin{proof}
By definition of the map $\KInd_K^G$  in Section 5 of \cite{Hochs09}, the left hand side of \eqref{eq KInd Dolb} is represented by a Dirac operator on 
 the bundle
\[
\xymatrix{
G\times_K\bigl( (K\times_T \Bigwedge_{\C} \kk/\kt \otimes \C_{\lambda - \rho}) \otimes \C_{\rho_n} \otimes S_{\kp} \bigr) \ar@{=}[r]^-{\sim} \ar[d]&
G\times_T  \bigl(\Bigwedge_{\C} \kk/\kt \otimes \C_{\lambda - \rho} \otimes  (S_{\kp} \otimes \C_{\rho_n}) \bigr)\ar[d]\\
G\times_K(K/T) \ar@{=}[r]^-{\sim} & G/T.
}
\]
Here $\Bigwedge_{\C}$ denotes the exterior algebra of complex vector spaces, and we used that $\C_{\rho_n} \otimes S_{\kp}$ is always a well-defined representation of $K$, 
see the bottom of page 21 in \cite{Atiyah77}. 
%

The operator $\bar \partial_{\lambda - \rho}$ acts on sections of the bundle
\[
G\times_T \Bigwedge_{\C} \kg/\kt \otimes \C_{\lambda - \rho} = G\times_T \Bigwedge_{\C} \kk/\kt \otimes \C_{\lambda -\rho} \otimes  \Bigwedge_{\C} \kp.
\]
Here  we used the fact that $\kp \hookrightarrow \kg/\kt$ is a complex subspace.
Therefore, it is enough to show that
\[
S_{\kp} \otimes \C_{\rho_n} = \Bigwedge_{\C} \kp
\]
as  representations of $T$, with the same grading if $\dim(G/K)/2$ is even, and the opposite grading if $\dim(G/K)/2$ is odd. 

On page 10 of \cite{Parthasarathy72}, it is noted that the set of weights of the representation of $T$ in $S_{\kp}$ is
\[
\Bigl\{ \frac{1}{2} \sum_{\alpha \in R^+}\varepsilon_{\alpha} \alpha; 
\varepsilon_{\alpha} = \pm 1 \Bigr\}. 
\]
The multiplicity of each weight is the number of ways it can be written in the above form.
The grading is according to the parity of the number of $\alpha$ with $\varepsilon_{\alpha} = -1$. So the weights of the representation of $T$ in $S_{\kp} \otimes \C_{\rho_n}$ are
\[
\Bigl\{
\frac{1}{2} \sum_{\alpha \in R^+} (\varepsilon_{\alpha}+1) \alpha; \varepsilon_{\alpha} = \pm 1 
\Bigr\}
\]
Now
\[
\frac{1}{2} \sum_{\alpha \in R^+} (\varepsilon_{\alpha}+1) \alpha = \sum_{\alpha \in A} \alpha, 
\]
where
\begin{equation} \label{eq S R+}
A = \{\alpha \in R^+; \varepsilon_{\alpha} = 1\}.
\end{equation}

The set of $T$-weights of $\Bigwedge_{\C} \kp$ is
\[
\Bigl\{
 \sum_{\alpha \in A} \alpha; A \subset R^+
\Bigr\}.
\]
 The multiplicity of each weight is the number of ways it can be written in this form. We conclude that the representations $S_{\kp} \otimes \C_{\rho_n}$ and $\Bigwedge_{\C} \kp$ have the same set of weights, and that each weight occurs with the same multiplicity.
 
 The grading on $\Bigwedge_{\C} \kp$ is according to the parity of exterior degrees. These degrees correspond to the number of elements of $S$ is the above expression for the weights. Under the correspondence \eqref{eq S R+}, this number of elements corresponds to the number of $\alpha \in R^+$ for which $\varepsilon_{\alpha} = 1$. If $\#R^+ = \dim(G/K)/2$ is even, this equals the number of $\alpha$ for which $\varepsilon_{\alpha} = -1$.
 So in that case, we have $S_{\kp} \otimes \C_{\rho_n} = \Bigwedge_{\C} \kp$ as graded representations of $T$.  If $\dim(G/K)/2$ is odd, then the parity of the number of elements of $S$ is the opposite of the parity of the number of $\alpha \in R^+$ for which $\varepsilon_{\alpha} = -1$ under the correspondence \eqref{eq S R+}. Then $S_{\kp} \otimes \C_{\rho_n}$ is isomorphic to  $\Bigwedge_{\C} \kp$ with its grading reversed, as a graded representation of $T$.
\end{proof}

\begin{proofof}{Proposition \ref{prop ds index}}
Successively applying Lemma \ref{lem ds DInd}, the Borel--Weil theorem, Theorem \ref{thm Q Ind},  and Lemma \ref{lem KInd Dolb},  we find that
\[
\begin{split}
[d_{\pi}m_{\xi}] &=\DInd_K^G[V_{\lambda - \rho_c}] \\
	&=  \DInd_K^G \bigl(\ind_K(\bar \partial^{K/T}_{L_{\lambda - \rho_c}} + (\bar \partial^{K/T}_{L_{\lambda - \rho_c}})^*) \bigr) \\
	&=  \ind_G\bigl(\KInd_K^G[\bar \partial^{K/T}_{L_{\lambda - \rho_c}} + (\bar \partial^{K/T}_{L_{\lambda - \rho_c}})^*]\bigr)\\
&= (-1)^{\dim(G/K)/2}\ind_G(\bar\partial_{L_{\lambda-\rho}} + \bar\partial_{L_{\lambda-\rho}}^*).
\end{split}
\]
\end{proofof}

\begin{remark} \label{rem indep}
It is of course important that we did not implicitly use Harish-Chandra's character formula to prove Corollary \ref{cor char}. We used the following results from representation theory, all of which are independent of the character formula. 
\begin{itemize}
\item The properties of Harish-Chandra's Schwartz algebra $\cC(G)$, which are part of general theory that does not involve results about the discrete series. (Except for the fact that $\cC(G)$ contains $K$-finite matrix coefficients of discrete series representations, which does not rely on information about characters.)
\item Existence of discrete series representations with given infinitesimal characters, under the condition that $\rank(G)=\rank(K)$. See Theorem 9.20 in \cite{Knapp}. (We did not use the necessity of the condition that $\rank(G)=\rank(K)$, or the exhaustion of the discrete series.)
\item Proposition 14.4.3 in \cite{Dixmier}, which is proved directly there. (This is used in the proof of Proposition \ref{prop taug char}.)
\item Lemmas 2.1.1 and  2.2.1 in \cite{Lafforgue02}. These are independent of Harish-Chandra's work. These lemmas are steps in Lafforgue's  independent proofs of the exhaustion of the discrete series, and of the necessary condition that $\rank(G)=\rank(K)$ for $G$ to have a discrete series. (These lemmas are used in the proof of Lemma \ref{lem ds DInd}.)
\item The fact that the Connes--Kasparov conjecture is true for connected semisimple Lie groups. Wassermann's proof in \cite{Wassermann87} (for linear reductive groups) was based on representation theory, including the classification of discrete series representations, but Lafforgue's proof in \cite{Lafforgue02} does not involve any knowledge about the discrete series. (This is used in the proof of Lemma \ref{lem ds DInd}.)
\end{itemize}
\end{remark}

\begin{small}

\bibliographystyle{plain}
\bibliography{mybib}

\end{small}

\end{document}